\documentclass[10pt,a4paper,reqno]{amsart}
\usepackage{amsfonts}
\usepackage{amssymb}
\usepackage{amsmath}
\usepackage{amsthm}
\usepackage{cite}
\usepackage{enumitem}
\usepackage{ulem}
\usepackage{tikz}
\usepackage{subfigure}
\usepackage{stackrel}
\usepackage{hyperref}
\theoremstyle{plain}
\newtheorem{thm}{Theorem}

\newtheorem{lem}[thm]{Lemma}
\newtheorem{prop}[thm]{Proposition}
\newtheorem{cor}[thm]{Corollary}
\theoremstyle{remark}
\newtheorem{rem}[thm]{Remark}

\providecommand{\sm}{\setminus}
\providecommand{\N}{\mathbb{N}}
\providecommand{\R}{\mathbb{R}} 

\providecommand{\C}{\mathbb{C}}
\providecommand{\cF}{\mathcal{F}}

\providecommand{\cL}{\mathcal{L}}

\providecommand{\cV}{\mathcal{V}}
\providecommand{\eps}{\varepsilon}
\providecommand{\ov}{\overline}
\providecommand{\dx}{\,\mathrm{d}x}
\providecommand{\wto}{\rightharpoonup}

\providecommand{\les}{\lesssim}

\DeclareMathOperator{\supp}{supp}
\DeclareMathOperator{\sign}{sign}

\DeclareMathOperator{\loc}{loc}
\DeclareMathOperator{\curl}{curl}
\DeclareMathOperator{\Real}{Re}
\DeclareMathOperator{\Imag}{Im}

\DeclareMathOperator{\dist}{dist}

\renewcommand{\qed}{\hfill $\Box$}

\newcommand{\vecII}[2]{
\ensuremath{
\begin{pmatrix}
#1 \\ #2 \\
\end{pmatrix}}}

\newcommand{\matII}[4]{
\ensuremath{
\begin{pmatrix}
#1 & #2 \\
#3 & #4 \\
\end{pmatrix}}}

\usepackage{color}
\definecolor{Darkgblue}{rgb}{0.3,0.3,0.5}

\begin{document}

\allowdisplaybreaks

\title[A Limiting Absorption Principle for Helmholtz systems and Maxwell's equations]{A Limiting Absorption
Principle for Helmholtz systems and time-harmonic isotropic Maxwell's equations}

\author{Lucrezia Cossetti, Rainer Mandel\textsuperscript{1}}
\thanks{Funded by the Deutsche Forschungsgemeinschaft (DFG, German Research Foundation)
- Project-ID 258734477 - SFB 1173.}
\address{\textsuperscript{1}Karlsruhe
Institute of Technology, Institute for Analysis, Englerstra{\ss}e 2, 76131 Karlsruhe, Germany}
\email{lucrezia.cossetti@kit.edu; rainer.mandel@kit.edu}
  

\date{\today}

\newpage

\begin{abstract}
  In this work we investigate the $L^p-L^q$-mapping properties of the resolvent associated with
  the time-harmonic isotropic Maxwell operator. As spectral parameters close to the spectrum are also covered
  by our analysis, we obtain an $L^p-L^q$-type Limiting Absorption Principle for this operator. Our
  analysis relies on new results for Helmholtz systems with zero order non-Hermitian perturbations. 
  Moreover, we provide an improved version of the Limiting Absorption Principle for Hermitian (self-adjoint)
  Helmholtz systems.
\end{abstract}

\maketitle
\allowdisplaybreaks

 \section{Introduction}

 The propagation of electromagnetic waves in continuous three-dimensional media is governed by the
 \textit{Maxwell's equations}. They consist of four equations, two vectorial and two scalar ones in the
 unknowns $\mathcal D$ and $\mathcal E$ (the electric fields) and $\mathcal B$ and $\mathcal H$ (the magnetic
 fields) for a given current density $\mathcal J$. Assuming the absence of electric charges their macroscopic
 formulation reads as follows
\begin{equation}\label{eq:Maxwell-time}
	\partial_t \mathcal D-\nabla \times \mathcal H=-\mathcal J, \qquad
	\partial_t \mathcal B+\nabla \times 	\mathcal E=0, \qquad
	\nabla \cdot \mathcal D=\nabla \cdot \mathcal B=0,
\end{equation} 
with $\mathcal D,\mathcal H, \mathcal B,\mathcal E,\mathcal J\colon \R \times \R^3 \to \C^3.$ 
Notice that the restriction to the case of no electric charges ($\rho=0$) influencing the propagation of the electromagnetic waves or, which is the same, the divergence-freeness of $\mathcal D$ and $\mathcal B,$ 
implies $\nabla \cdot \mathcal J=0.$ 
Constitutive relations that specify  the connections between the electric displacement $\mathcal D$ and the
electric field $\mathcal E$ and between the magnetic flux density $\mathcal B$ and the magnetic field
$\mathcal H$ are necessary for meaningful applications of this model. In general,  
these relations need not be simple, but in the physically realistic scenario where ferro-electric
and ferro-magnetic materials are discarded and where the fields are weak enough, the material laws may  be
assumed to obey the following \textit{linear} relations:
\begin{equation}\label{eq:lin-constitutive-rel}
\mathcal D=\varepsilon(x) \mathcal E,\qquad
\mathcal B=\mu(x) \mathcal H.
\end{equation}  
Here $\varepsilon$ and $\mu$ embody  the \textit{permittivity} respectively the \textit{permeability} of the
medium. In general anisotropic materials, where the interaction of fields and matter
not only depends on the position in the material but also on the direction of the fields, these
quantities are mathematically represented as tensors.  In this paper we will be exclusively concerned  with
the case of \textit{isotropic} (\textit{i.e.} direction-independent) media where $\varepsilon$ and $\mu$
are \textit{scalar}-valued functions on $\R^3.$ For a more detailed description of   Maxwell's equations we
refer the reader to~\cite{Jackson,Kuchment}.

\medskip

We will focus on \textit{monochromatic} waves only, \textit{i.e.}, electromagnetic fields
$\mathcal E, \mathcal D, \mathcal B, \mathcal H, \mathcal J$ that are periodic functions of time with the same frequency $\omega\in
\R\sm\{0\},$ more specifically $\mathcal E(x,t):=e^{i\omega t}E(x)$, $\mathcal D(x,t):=e^{i\omega
t}D(x)$, $\mathcal B(x,t):=e^{i\omega t}B(x)$, $\mathcal H(x,t):=e^{i\omega t}H(x)$, $\mathcal
J(x,t):=e^{i\omega t}J(x)$ for vector fields $E,D, B,H, J\colon \R^3 \to \C^3.$ This gives rise to
the following time-harmonic analogue of Maxwell's equations~\eqref{eq:Maxwell-time} once the linear
constitutive relations from~\eqref{eq:lin-constitutive-rel} are imposed:
\begin{equation}\label{eq:Maxwell-harmonic-simplified}
i\omega \eps E - \nabla \times H=-J,\qquad
i\omega \mu H + \nabla \times E=0.
\end{equation}
In this paper we are interested in the following slightly more general model 
\begin{equation}\label{eq:Maxwell-harmonic}
	i\zeta \eps E - \nabla \times H=-J_\textup{e}, \qquad
	i\zeta \mu H + \nabla \times E=J_\textup{m},
\end{equation}
where $\zeta\in \C$ and where both electric and magnetic current densities $J_\textup{e}$ and $J_\textup{m}$
are included. Allowing for spectral parameters $\zeta\in\C\sm\R$ reflects the so-called Ohm's law for
conducting media, which asserts that the current $J$ induced by the electric field $E$ can
be described (in linear approximation) by $J = \sigma E + J_\textup{e}$,  
where $\sigma\colon \R^3\to \R$ represents the \textit{conductivity} and $J_\textup{e}$ is  the external current
density. Thus, plugging in Ohm's law into~\eqref{eq:Maxwell-harmonic-simplified} one gets that the first
equation in~\eqref{eq:Maxwell-harmonic-simplified} can be rewritten as
\begin{equation*}
	i(\omega \eps -i\sigma)E- \nabla \times H=-J_\textup{e},
\end{equation*}
which motivates the interest in the model~\eqref{eq:Maxwell-harmonic}.

\medskip

The main purpose of this paper is to prove an $L^p$-type Limiting Absorption Principle for the
time-harmonic Maxwell's equations~\eqref{eq:Maxwell-harmonic}. Roughly speaking, proving a Limiting Absorption
Principle means proving existence and continuity of the resolvent operator up to the essential spectrum. In the context of
the Maxwell system~\eqref{eq:Maxwell-harmonic} this translates into studying the boundedness of
solutions $(E_\zeta, H_\zeta)$ of~\eqref{eq:Maxwell-harmonic} with $\Imag(\zeta)\neq 0$ and characterizing
their limits as $\Imag(\zeta) \to 0^\pm.$ In this paper we shall prove the following result.

\begin{thm}\label{thm:main}
	Let $\omega\in\R\sm\{0\}$ and assume that $1\leq p,\tilde p, q\leq \infty$ satisfy
\begin{equation} \label{eq:LAP_Maxwell_conditions} 
  \frac{2}{3}<\frac{1}{p}<1,\qquad
  \frac{1}{6}<\frac{1}{q}<\frac{1}{3},\qquad
  \frac{1}{2}\leq \frac{1}{p}-\frac{1}{q}\leq \frac{2}{3},\qquad
  0\leq \frac{1}{\tilde p}-\frac{1}{q}\leq \frac{1}{3}.   
\end{equation}
 Moreover assume  that there are $\eps_\infty,\mu_\infty>0$ such that
\begin{enumerate}
	\item[(A1)] $\eps, \mu \in W^{1, \infty}(\R^3)$  are uniformly positive, 
	\item[(A2)]  $|\nabla (\eps \mu)| + |\eps\mu -\eps_\infty \mu_\infty| + 
  	|\nabla \eps|^2 + |\nabla \mu|^2+
  	|D^2 \eps| + |D^2 \mu|\in L^{\frac{3}{2}}(\R^3) + L^2(\R^3)$, 
    \item[(A3)]  $|\nabla(\eps\mu)|+|\eps\mu-\eps_\infty\mu_\infty|\in
    L^\frac{q_1}{q_1-2}(\R^3)+L^{\frac{q_2}{q_2-2}}(\R^3)$ where $q\in [q_1,q_2]$ and 
    $(p,\tilde p,q_1),(p,\tilde p,q_2)$ satisfy~\eqref{eq:LAP_Maxwell_conditions}.
\end{enumerate}
   Then for all divergence-free vector fields $J_\textup{e}, J_\textup{m}\in L^p(\R^3;
  \C^3)\cap L^{\tilde p}(\R^3;\C^3)$, there are weak solutions
  $(E_\omega^\pm,H_\omega^\pm)\in L^q(\R^3;\C^6)\cap H^1_{\loc}(\R^3;\C^6)$ of the time-harmonic Maxwell
  system~\eqref{eq:Maxwell-harmonic} with $\zeta=\omega$ satisfying 
  \begin{equation}\label{eq:resolvent-estimate}
	  \|(E_\omega^\pm, H_\omega^\pm)\|_q\leq C(\omega) \big( \|(J_\textup{e}, J_\textup{m})\|_p +
	  \|(J_\textup{e}, J_\textup{m})\|_{\tilde p}\big) 
	\end{equation}
  where $\omega\mapsto C(\omega)$ is continuous on $\R\sm\{0\}$. Moreover the following holds:
	\begin{itemize}
	  \item[(i)] We have $(E_\zeta,H_\zeta)\to (E_\omega^\pm,H_\omega^\pm)$ in $L^q(\R^3;\C^6)\cap
	  H^1_{\loc}(\R^3;\C^6)$ as $\zeta\to \omega\pm i0$ where $(E_\zeta,H_\zeta)\in H^1(\R^3;\C^6)$ is the
	  unique weak solution solution of~\eqref{eq:Maxwell-harmonic} with divergence-free vector fields
	  $J_\textup{e}^\zeta,J_\textup{m}^\zeta\in L^p(\R^3;\C^3)\cap L^{\tilde p}(\R^3;\C^3)\cap L^2(\R^3;\C^6)$
	  converging to $J_\textup{e},J_\textup{m}$ in $L^p(\R^3;\C^3)\cap L^{\tilde p}(\R^3;\C^3)$, respectively. 
	  \item[(ii)] The function $u_\omega^\pm := (\eps^\frac{1}{2} E_\omega^\pm,\mu^\frac{1}{2}H_\omega^\pm)\in
	  L^q(\R^3;\C^6)$ solves the Helmholtz system
	  $$
	    (\Delta+ \omega^2\eps_\infty\mu_\infty) u_\omega^\pm  + \cV(\omega)u_\omega^\pm 
	    = \cL_1(\omega)\tilde J+\cL_2\tilde J
	  $$
	  where $\cV(\omega),\cL_1(\omega),\cL_2,\tilde J$ are defined at the beginning of
	  Section~\ref{section:absenceEV}. More precisely, $u_\omega^\pm$ satisfies the  
	  integral equation~\eqref{eq:EquationForSolution}.
	  \item[(iii)] If additionally $J_\textup{e},J_\textup{m}\in L^q(\R^3;\C^3)$, then
	  $(E_\omega^\pm,H_\omega^\pm)\in W^{1,q}(\R^3;\C^6)$.
	\end{itemize}
\end{thm}

\medskip

\begin{rem} ~
  \begin{itemize}
    \item[(a)] We shall not provide explicit values for the constants $C(\omega)$, but content ourselves with
    proving estimates that are uniform on compact subsets of $\R\setminus \{0\}$. This is indeed sufficient
    for the existence of a map $\omega\mapsto C(\omega)$ that is continuous on $\R\sm\{0\}$. 
    \item[(b)] The convergence in  $H^1_{\loc}(\R^3;\C^6)$ is stated for simplicity. By
    standard elliptic regularity theory, convergence holds in $W^{2,r}_{\loc}(\R^3;\C^6)$ where $r\geq 1$
    depends on the local regularity of $\eps,\mu,J_\textup{e},J_\textup{m}$.  
    \item[(c)] If the currents $J_\textup{e}^\zeta,J_\textup{m}^\zeta$ also converge in $L^q(\R^3;\C^3)$ then
     one finds $(E_\zeta,H_\zeta)\to (E_\omega^\pm,H_\omega^\pm)$ in $W^{1,q}(\R^3;\C^6)$.
  \end{itemize}
\end{rem}

In the context of Limiting Absorption Principles of time-harmonic Maxwell's equations only few results are
available. Picard, Weck and Witsch proved a Limiting Absorption Principle in weighted $L^2$-spaces (similar
to~\cite{Agmon_Spectral}) for time-harmonic Maxwell's equations in an exterior domain with boundary conditions $\nu \wedge E = 0$, see~\cite[Theorem 2.10]{PicWecWeit_THMaxwell}.
Since this result is based on Fredholm's Alternative, the frequencies $\omega\in\R\sm\{0\}$ are assumed not to
belong to a discrete (possibly empty) set of eigenvalues. As in Agmon's fundamental paper~\cite{Agmon_Spectral} about the
perturbed Helmholtz equation, the permittivity $\eps$ and permeability $\mu$ are assumed
to be isotropic and to decay to some positive constants at infinity faster than $|x|^{-1}$. Despite some
quantitative differences, this is similar to our assumptions (A1),(A2),(A3). 
The Limiting Absorption Principle in weighted $L^2$-spaces and applications to Strichartz estimates for
Maxwell's Equations in isotropic media were recently proved by D'Ancona and Schnaubelt~\cite{DAnSchnaubelt}.

In the anisotropic case, still in weighted $L^2$-spaces, related results were obtained by
Pauly~\cite[Theorem~3.5]{Pauly_Low}. We note that these results also apply to discontinuous $\eps,\mu$, which
indicates that (A1) may be relaxed. For results in the $L^p-L^q$-setting as in this paper we refer
to~\cite{Ski,MaSki}. 
In a recent work~\cite[Theorem 2.1]{Nguyen_2020} by Nguyen and Sil,
in an $L^2$-framework, the Limiting Absorption Principle is studied in the case
of anisotropic sign-changing coefficients on $\R^3$ that are used to describe metamaterials.
We also mention~\cite[Lemma 4-Lemma 6]{Nguyen2018} where the author studied existence, uniqueness and stability of solutions to the anisotropic Maxwell system when the coefficients are constant outside a bounded set and piecewise $C^1(\R^3)$.

 As far as we can see, our contribution is
the first dealing with $L^p$-estimates for time-harmonic Maxwell's equations. All the aforementioned results
relate to the three-dimensional Maxwell operator. A Limiting Absorption Principle in periodic 1D waveguides
can be found in the recent preprint by De Nittis, Moscolari, Richard and Tiedra de Aldecoa~\cite{nittis2019spectral}. 
Further relevant tools for Limiting Absorption Principles such as Carleman inequalities or Unique Continuation
results can be found in~\cite{Okaji_SUCP,EllerYamamoto_Carleman}.

\medskip

In view of the Limiting Absorption Principle from Theorem~\ref{thm:main} and in particular of the
resolvent-type estimate~\eqref{eq:resolvent-estimate}, it must be expected that embedded eigenvalues of the
Maxwell operator do not exist under the assumptions of Theorem~\ref{thm:main}. In the classical Fredholm
theoretical approaches from~\cite{Agmon_Spectral,PicWecWeit_THMaxwell} this is even a necessary condition for
the Limiting Absorption Principle to hold. Nonetheless, the proof of our Theorem~\ref{thm:main} does not
allow to derive the absence of embedding eigenvalues directly.  
For this reason we prove the absence of
embedded eigenvalues separately. This is object of the following result. As we shall see in
Section~\ref{section:absenceEV}, this heavily relies on Carleman estimates by Koch and
Tataru~\cite{KochTat_Absence}.

\begin{thm}\label{thm:absenceEV}
   Assume (A1),(A2) for some $\eps_\infty, \mu_\infty>0$, $\zeta\in \C$  and let $(E,H)\in
   H^1_{\loc}(\R^3;\C^6)$ be a weak solution of the homogeneous ($J_\textup{e}=J_\textup{m}=0$) time-harmonic Maxwell
   system~\eqref{eq:Maxwell-harmonic} that satisfies $(1+|x|)^{\tau_1 -\frac{1}{2}}(|E|+|H|)\in
   L^2(\R^3)$ for some $\tau_1>0$. Then $E\equiv H\equiv 0$.
\end{thm}

In relation with non-existence of eigenvalues of the Maxwell operator, we should mention the remarkable
contribution of Eidus~\cite{EidusMaxwell}.
Here the author proves that for sufficiently smooth, real-valued, symmetric and positive definite matrices
$\varepsilon$ and $\mu$ there are no nontrivial solutions of the homogeneous ($J_\textup{e}=J_\textup{m}=0$)
time-harmonic Maxwell system~\eqref{eq:Maxwell-harmonic} with $\zeta\in \R\setminus \{0\}$ imposing that the
matrix-valued coefficients $\varepsilon$ and $\mu$ are short-range perturbations of the identity and satisfy
the repulsivity condition
\begin{equation} \label{eq:repulsive-cond}
	\langle\partial_r (r \varepsilon) u, u \rangle_{L^2}\geq \gamma \|u\|_{L^2}^2, \qquad r=|x|,
\end{equation}
for some positive $\gamma$ (and analogous condition for $\mu$). We stress that in~\cite{EidusMaxwell} the necessity of condition~\eqref{eq:repulsive-cond} for the absence of bound states is not discussed, as a matter of fact in~\cite{EidusMaxwell} condition~\eqref{eq:repulsive-cond} is only introduced as a \textit{sufficient} condition. We refer the reader to~\cite[Theorem
4.2]{EidusMaxwell} for details. In the same paper an analogous result in the
isotropic case is proved, which shares some similarity with Theorem~\ref{thm:absenceEV}. In this case it
turns out that the nonexistence of eigenvalues follows without the need of the repulsive
condition~\eqref{eq:repulsive-cond} stated above (see~\cite[Theorem 4.4]{EidusMaxwell}).
In the proofs of Theorem~\ref{thm:main} and~\ref{thm:absenceEV} we will use that for any given solution
$(E,H)$ of the Maxwell system~\eqref{eq:Maxwell-harmonic} the function $(\tilde E,\tilde H):=
(\eps^\frac{1}{2} E,\mu^\frac{1}{2} H)$ solves a linear Helmholtz system with complex-valued zeroth order
perturbations. The difficulty arises from the fact this perturbation is non-Hermitian in general.  The tools
that we will need in the analysis of this particular system are inspired from the theory for Helmholtz systems with Hermitian perturbations that we will develop first.
For the sake of simplicity we restrict our attention to the case $n\geq 3$.

\begin{thm}\label{thm:main-Helmholtz-sys}
	Let $n,m\in \N, n\geq 3,$ $\zeta\in \C\sm\R.$ Assume $V=\ov{V}^T\in L^\frac{n}{2}(\R^n;\C^{m\times m})+
	L^{\frac{n+1}{2}}(\R^n;\C^{m\times m})$ and that $1\leq p, q\leq \infty$ satisfy
	$$
	\frac{n+1}{2n}<\frac{1}{p}\leq 1,\qquad
   \frac{(n-1)^2}{2n(n+1)} <\frac{1}{q}<\frac{n-1}{2n},\qquad
\frac{2}{n+1}\leq \frac{1}{p}-\frac{1}{q}\leq \frac{2}{n}.  
	$$ 
	Then $R(\zeta):= (\Delta I_{m} + V(x) + \zeta I_{m})^{-1}\colon L^p(\R^n;\C^m)\to L^q(\R^n;\C^m)$ exists
	as a bounded linear operator and extends by pointwise convergence to the positive half-axis via
	\begin{equation} \label{eq:continuousExtension}
	   R(\lambda\pm i0)f :=  \lim_{\zeta\to \lambda\pm i0} R(\zeta)f, 
	  \qquad \text{in}\; L^q(\R^n;\C^m),
	\end{equation} 
	for any $f\in L^p(\R^n;\C^m)$ and $\lambda>0$. 
\end{thm} 

\medskip

 \begin{rem} ~
   \begin{itemize} 
     \item[(a)] We will actually prove a slightly stronger result than Theorem~\ref{thm:main-Helmholtz-sys}.
     We will show that all conclusions mentioned in this theorem are true assuming
     $V=\ov{V}^T\in L^\frac{n}{2}(\R^n;\C^{m\times m})+L^{\tilde\kappa}(\R^n;\C^{m\times m})$ with
     $\frac{n}{2}\leq \tilde\kappa\leq\frac{n+1}{2}$ where the  condition $\frac{(n-1)^2}{2n(n+1)}
     <\frac{1}{q}<\frac{n-1}{2n}$ is replaced by the weaker one $\frac{n+1}{2n}-\frac{1}{\tilde\kappa}
     <\frac{1}{q}<\frac{n-1}{2n}$. In other words, Theorem~\ref{thm:main-Helmholtz-sys} corresponds to the
     special case $\tilde\kappa=\frac{n+1}{2}$ of this stronger result.
     \item[(b)] 
     We assume $\zeta\in\C\sm\R$ in order to guarantee the existence of the resolvent $R(\zeta)$. 
     Notice that sign or smallness assumptions on $V$ may ensure the existence of well-defined resolvents on 
     given parts of the real line. 
      \item[(c)] The limit in~\eqref{eq:continuousExtension} is a pointwise limit and it is natural to ask
     whether this convergence also holds in the uniform operator topology. Ideas related to this question can
     be found in \cite[p.46]{HuangYaoZheng}.  
     \item[(d)] The two-dimensional case $n=2$ can in principle be discussed using the same techniques. We
     expect that the same statements hold for $V=\ov{V}^T \in L^\kappa(\R^2;\C^{m\times
     m})+L^{\frac{3}{2}}(\R^2;\C^{m\times m})$ for some $\kappa>1$.
   \end{itemize} 
 \end{rem}

To explain  to what extent Theorem~\ref{thm:main-Helmholtz-sys} improves 
earlier results, we first provide a short summary of the available literature
about the scalar case $m=1$.
Goldberg and Schlag~\cite{GolSch_LAP} were the first to go beyond the Hilbert space framework in which, since
Agmon's work~\cite{Agmon_Spectral}, the Limiting Absorption Principles for self-adjoint Schr\"odinger operators were studied. They proved an $L^p$-type Limiting Absorption Principle that inspired our Theorem~\ref{thm:main-Helmholtz-sys}. For $n=3$ they showed
\begin{equation}\label{eq:G-S}
	\sup_{0<\delta<1,\, \lambda\geq \lambda_0}\|R(\lambda +i\delta) \|_{\frac{4}{3}\to 4}\leq C(\lambda_0, V) \lambda^{-\frac{1}{4}}, \qquad \lambda_0>0,
\end{equation}
provided that $V\in L^\frac{3}{2}(\R^3)\cap L^p(\R^3),$ $p>\frac{3}{2}.$  
In \cite[Proposition~1.3]{GolSch_LAP} it is even stated that the assumption on $V$ can be considerably
weakened to $V\in L^{\frac{3}{2}+\eps}(\R^3)+ L^{2-\eps}(\R^3), \eps>0$, provided that embedded
eigenvalues for Schr\"odinger operators with such pertubations do not exist. The latter was meanwhile proved by Koch and
Tataru~\cite[Theorem~3]{KochTat_Absence}.
Huang, Yao and Zheng ~\cite{HuangYaoZheng} generalized the result from~\cite{GolSch_LAP}
to the higher-dimensional case by proving that the estimate
\begin{equation*}
	\sup_{0<\delta<1, \lambda\geq\lambda_0}\|R(\lambda +i\delta) \|_{\frac{2(n+1)}{n+3}\to
	\frac{2(n+1)}{n-1}}\leq C(\lambda_0,V) \lambda^{-\frac{1}{n+1}}, \qquad \lambda_0>0,
\end{equation*}
holds for all potentials $V\in L^\frac{n}{2}(\R^n)\cap L^p(\R^n),$ $p>\frac{n}{2}$ and all
$n\in\N,n\geq 3$.
The most recent result in this direction is due to Ionescu and Schlag~\cite{IonSch_AgmonKatoKuroda} 
where a new Limiting Absorption Principle was proved for a much larger class of potentials than
the ones covered by the aforementioned results. As our Theorem~\ref{thm:main-Helmholtz-sys}, this 
result covers all  $V\in L^{\frac{n}{2}}(\R^n)+L^{\frac{n+1}{2}}(\R^n)$ as one can check 
from (1.19) in~\cite{IonSch_AgmonKatoKuroda}. In their Theorem~1.3~(d), the  resolvent
estimate 
\begin{equation*} 
	\sup_{\lambda \in I,\, 0<\delta\leq 1} \|R(\lambda \pm i \delta)\|_{X\to X^*}\leq C(I,V)
\end{equation*}    
is proved where $I\subset \R \setminus \{0\}$ is a compact set that does not intersect 
the set of nonzero eigenvalues. For the precise definition of the Banach space $X$   we refer
to~\cite[p.400]{IonSch_AgmonKatoKuroda}. Notice that these estimates are
self-dual in the sense that they bound the operator norms of the resolvents acting between some Banach space
and the corresponding dual space. In this respect, our result from Theorem~\ref{thm:main-Helmholtz-sys} is
more general than~\cite{IonSch_AgmonKatoKuroda}. Compared to~\cite{HuangYaoZheng,GolSch_LAP},
Theorem~\ref{thm:main-Helmholtz-sys}  requires for less integrability of the potential (including endpoint
cases) and therefore improves the known results. Being given~\cite[Theorem~3]{KochTat_Absence}, the
generalization to Helmholtz systems ($m\geq 2$) is rather trivial.
 
Concerning Limiting Absorption Principles for Helmholtz equations ($m=1$)
in other settings and under different assumptions we would like to mention the papers
\cite{CacDAncLuc_LAP1,CacDAncLuc_LAP2} (Morrey-Campanato spaces) and  
\cite{Royer_LAP} for dissipative Helmholtz operators, \cite{Nguyen_LAP} (sign-changing coefficients), 
\cite{Man_LAP,BirmanYafaev_Scattering,Radosz_LAP} (periodic potentials) and 
\cite{BoucletMizutani_Uniform,MizZhaZhe_Uniform} (critical potentials).

\medskip

As anticipated (see again
Proposition~1.3 in~\cite{GolSch_LAP}), in the proof of a Limiting Absorption Principle, excluding embedded
eigenvalues usually represents the discriminating step where the hypotheses on the potential come
into play. The non-existence of positive eigenvalues for the Schr\"odinger operator has a long history. 
In 1959, Kato~\cite{Kato} proved that the Schr\"odinger operator $-\Delta + V$ in $L^2(\R^n)$ has no embedded
eigenvalues if $V$ is continuous and such that $V(x)=o(|x|^{-1})$ as $|x|\to \infty.$ Later
Simon~\cite{Simon} improved Kato's result allowing also for long-range potentials. More precisely he
considered potentials V which admit the decomposition $V=V_1+ V_2,$ with $V_1(x)=o(|x|^{-1}), V_2(x)=o(1),$
and $\omega_0:=\limsup_{|x|\to \infty} x\cdot \nabla V_2(x)<\infty.$ Under these conditions he
proved, in three dimensions, the absence of eigenvalues above the positive threshold $\omega_0.$ This
result was later improved by Agmon in~\cite{Agmon70}, where he lowered the threshold to $\omega_0/2$
covering also all dimensions. Similar results were also proved by Froese et al. in~\cite{Froese2} (see
also~\cite{FKV1,FKV2, CFK, AHK} for related results). Using a different approach based on Carleman
estimates Ionescu-Jerison in~\cite[Theorem 2.5]{IonJer_OnTheAbsence} showed that for all $\eps>0$ there are
$V\in L^{\frac{n+1}{2} + \eps}(\R^n)$ such that the scalar Schr\"odinger operator $\Delta+V$ has embedded
eigenvalues with rapidly decaying eigenfunctions. Thus, the exponent $\frac{n+1}{2}$ in our assumption is
optimal. Notice that, as far as asymptotic decay conditions are investigated, a higher exponent in the
Lebesgue space allows to cover a wider class of perturbations. On the other hand, the optimality of the
exponent $\frac{n}{2}$ is not entirely clear, even though it is known that standard properties of
Schr\"odinger operators like semi-boundedness need not hold for potentials with lower integrability.
In~\cite{KenNad_Counterexample},\cite[Theorem~1.a)]{KochTataru_Counterexamples} it is shown that 0 can be an
embedded eigenvalue when potentials in the class $L^{\kappa}_{\loc}(\R^n)\cap L^1(\R^n)$ with
$\kappa<\frac{n}{2}$ are considered, see also~\cite[Remark~6.5]{JerisonKenig}. Up to the authors' knowledge,
a counterexample for non-zero eigenvalues is not known.

\medskip

 As customary, a basic tool for ruling out embedded eigenvalues 
is a suitable Carleman estimate. In our case, due to the weak and almost optimal conditions $V=\ov{V}^T\in
L^\frac{n}{2}(\R^n;\C^{m\times m})+ L^\frac{n+1}{2}(\R^n;\C^{m\times m}),$ we need to use the fine Carleman
estimate for scalar Schr\"odinger operators provided by Koch and Tataru in~\cite{KochTat_Absence}, which
allows to cover this wide class of potentials. We stress that the possibility to use a scalar Carleman
estimate in our vector-valued setting only works because the chosen weight in the Carleman bound
in~\cite[Proposition~4]{KochTat_Absence} \textit{does not} depend on the solution itself. Indeed, this fact ultimately permits to sum up the  estimates obtained for the
components and to get an estimate for the full vector field.  
Analogue results for Helmholtz systems with first order perturbations  cannot  be obtained
in this way since the weights in the corresponding Carleman estimates 
from~\cite{KochTat_Absence} (see Theorem~8 and Theorem~11) depend  on the solution itself. Hence,  it is not
guaranteed that one Carleman weight works for all components, which is why  
systems with first order perturbations appear to be more difficult.

\medskip

The rest of the paper is organized as follows: Section~\ref{sec:LAP_Helmholtz} is devoted to the proof of the
Limiting Absorption Principle for Helmholtz systems with Hermitian coefficients stated in
Theorem~\ref{thm:main-Helmholtz-sys}. The aforementioned  relation  between  Maxwell's
equations~\eqref{eq:Maxwell-harmonic} and   Helmholtz systems will be discussed in
Section~\ref{section:absenceEV}. Here we also provide the proof of Theorem~\ref{thm:absenceEV} about the
absence of eigenvalues for the Maxwell system~\eqref{eq:Maxwell-harmonic}. Finally, in Section~\ref{sec:LAP},
we prove the most involved result of the paper, namely the Limiting Absorption Principle for
Maxwell's equations~\eqref{eq:Maxwell-harmonic} from Theorem~\ref{thm:main}.

\medskip

We conclude this introduction with the main notations used in this paper.

\medskip

\subsection*{Notations} ~ \\ 
	* For $Z\in \{\R,\C,\R^m,\C^m,\R^{m\times m},\C^{m\times m}\}$ we shall shortly write
    $\|\cdot\|_p:= \|\cdot\|_{L^p(\R^n;Z)}$.  \\
    * $\mathcal B(X,Y)$ denotes the Banach spaces of bounded linear operators between Banach spaces $X,Y$
    equipped with the standard operator norm.   \\
    * We write  $V\in L^{[p_1,p_2]}(\R^n;Z):= L^{p_1}(\R^n;Z) +
    L^{p_2}(\R^n;Z)$ if $V$ can be decomposed as $V=V_1+V_2,$ with $V_1\in L^{p_1}(\R^n;Z)$ and $V_2\in
    L^{p_2}(\R^n;Z).$\\
    * $\chi_B$ represents the indicator of a measurable subset $B\subset\R^n.$ \\
    * $\zeta\to \lambda\pm i0$ means $\zeta\to \lambda$ with $\pm \Imag(\zeta) \searrow 0$. \\
    * $I_m$ denotes the identity matrix in $\R^{m\times m},$ $m\in \N.$ \\
    * The notation $I$ is   used for the identity operator in some function space \\
    * We use the notation $\lesssim$ where we want to indicate that we have an inequality $\leq$ up to a
    constant factor which does not depend on the relevant parameters.\\
    * $H^1(\curl;\R^3):=\{u\in L^2(\R^3;\C^3)\colon \curl u \in L^2(\R^3;\C^3)\}.$\\
    * We adopt the following definition for the Fourier transform
    \begin{equation*}
	\widehat{f}(\xi):=\mathcal{F}(f)(\xi):=\frac{1}{(2\pi)^{n/2}}\int_{\R^n} e^{-ix\xi}f(x)\dx.
\end{equation*}

\section{The LAP for Helmholtz systems -- Proof of Theorem~\ref{thm:main-Helmholtz-sys}}
\label{sec:LAP_Helmholtz}

This section is concerned with the proof of Theorem~\ref{thm:main-Helmholtz-sys} that relies on a well-known
perturbative argument based on Fredholm operator theory. This strategy has its origin in the pioneering work
by Agmon~\cite{Agmon_Spectral} where it was used to establish a Limiting Absorption Principle for
Schr\"odinger operators acting between weighted $L^2$-spaces.  Since then, this technique
has permeated many works in the subject. 
We refer to~\cite{Sommerfeld,Rellich,Eidus} for some remarkable earlier contributions.

\medskip

\noindent 
 We summarize Agmon's approach as follows.
Consider a reference operator $H_0$ and let $H$ be a suitable perturbation  of $H_0$, let $\zeta\in\C$. The
first step is to prove the existence of a right inverse
 $R_0(\zeta)$ for the operator  $H_0+\zeta$ satisfying an estimate of the form
\begin{equation}\label{eq:resolvent-gen-free}
	\|R_0(\zeta) f\|_{X_1}\leq C(\zeta)\|f\|_{X_2}, 
\end{equation}  
where $X_1, X_2$ are Banach spaces. In Agmon's paper, for spectral parameters $\zeta:=\lambda>0$ and the
Laplacian $H_0=\Delta$ such right inverses are constructed via the classical Limiting Absorption Principle for Helmholtz
equations, namely by investigating the mapping properties of the resolvents $R_0(\zeta)$ as
 $\zeta\to\lambda\pm i0, \lambda>0$, see
Theorem~4.1~\cite{Agmon_Spectral}. This is a nontrivial task given that every such $\lambda$ belongs to the essential spectrum of the
(negative) Laplacian and therefore no such limits can exist when $X_1=X_2=L^2(\R^n).$ In
\cite{Agmon_Spectral,Agmon_ARepresentation} this was circumvented by introducing suitable and, as a
matter of fact, optimal weighted $L^2-$spaces such that the operators $R_0(\zeta)$ converge in 
$\mathcal L(X_2,X_1)$ as $\zeta\to\lambda\pm i0$  (with different limits). 
In order to extend the estimate~\eqref{eq:resolvent-gen-free} to the perturbed operator $H$ one assumes  that
$V:=H-H_0$ is a relatively compact perturbation of $H_0$, meaning that the linear operator $K(\zeta):=-R_0(\zeta)V$ is compact on
$X_1$. In view of the formula 
$$
	H+\zeta=(H_0 + \zeta) (I -K(\zeta))
$$
a right inverse $R(\zeta)$ for the operator $H+\zeta$ is given by  
$$
	R(\zeta) :=(I -K(\zeta))^{-1} R_0(\zeta) 
$$
as soon as $I - K(\zeta)\colon X_1 \to X_1$ is bijective. By Fredholm theory, it suffices to
verify injectivity, which is the most delicate part of the argument. Once this is
achieved, one obtains the desired estimate 
\begin{equation}\label{eq:resolvent-gen-perturbed}
	\|R(\zeta)f\|_{X_1}\leq C(\zeta) \|(I -K(\zeta))^{-1}\|_{X_1\to X_1} \|f\|_{X_2}.
\end{equation} 
We stress that a good control of the right hand side with respect to $\zeta$ will be of central interest in
the following.

\medskip

In our context the reference operator $H_0$ and its perturbation $H$ are the free and
the perturbed matrix-valued Schr\"odinger operators, namely
\begin{equation*}
	H_0:=\Delta , \qquad H:= \Delta   + V(x),
\end{equation*}	
where $V=\ov{V}^T\in L^{[\frac{n}{2},\frac{n+1}{2}]}(\R^n;\C^{m\times m})$ and $m\in\N$. Here, the Laplacian
$\Delta$ acts as a diagonal operator on each of the $m$ components and 
$\zeta\in\C\sm\R_{\geq 0}$ or $\zeta=\lambda\pm i0,\lambda>0$ as we explain below.

According to the general strategy described above, to get an analogue of
estimate~\eqref{eq:resolvent-gen-perturbed} under our assumptions, we need to accomplish the following three
steps:
\begin{enumerate}
	\item[] Step 1:\; Provide $L^p-L^q$ estimates for $R_0(\zeta)$. 
	\item[] Step 2:\; Show that the linear operator $K(\zeta)=-R_0(\zeta)V \colon L^q(\R^n;\C^m)\to
	L^q(\R^n;\C^m)$ is compact.
	\item[] Step 3:\; Prove the injectivity of the Fredholm operator $I-K(\zeta)\colon L^q(\R^n;\C^m)\to
	L^q(\R^n;\C^m)$.
\end{enumerate}
We will see that Step~1 is essentially available in the literature. Only minor modifications will be
needed to pass from the scalar to the vector-valued framework.  To accomplish Step~2, which is rather
standard, we will use the local compactness of Sobolev embeddings. So the main difficulty is to
achieve Step 3. It will be accomplished with the aid of Carleman estimates by Koch and
Tataru~\cite{KochTat_Absence} and by exploiting the fact that $V$ is Hermitian. \\
Our results from Theorem~\ref{thm:main-Helmholtz-sys} even provide the
uniform bounds in $\C\sm\R_{\geq 0}$ 
$$
  C(\zeta):=\|R_0(\zeta)\|_{p\to q}  \les |\zeta|^{\frac{n}{2}(\frac{1}{p} - \frac{1}{q} -
  \frac{2}{n})} \qquad\text{and}\qquad
  \|(I-K(\zeta))^{-1}\|_{q\to q} \les 1
$$
as well as continuity properties of $\zeta\mapsto K(\zeta)$ and $\zeta\mapsto (I-K(\zeta))^{-1}$ needed 
for the proof of~\eqref{eq:continuousExtension}.  
The following subsections are devoted to the proof of the aforementioned facts.

\medskip

\subsection{$L^p-L^q$ estimates for $R_0(\zeta)$}

In the \textit{scalar} case, optimal $L^p-L^q$ resolvent estimates for $n\geq 3$ are originally due to
Kenig, Ruiz and Sogge~\cite[Theorem~2.3]{K_R_S} in the selfdual case $q=p'$ and to Guti\'{e}rrez
in~\cite[Theorem~6]{Gut_nontrivial} in the general case. 
For the precise asymptotics with respect to $\zeta$, which results from rescaling, we refer
to~\cite[p.1419]{KwonLee_Sharp}.

\begin{thm}[Kenig-Ruiz-Sogge, Guti\'{e}rrez]\label{thm:Gutierrez}
Let $m=1,n\in \N, n\geq 3 $ and assume $\zeta \in \C\setminus \R_{\geq 0}$.
	Then, for $1\leq p,q\leq \infty$ such that
	\begin{equation}\label{eq:indices-Gutierrez}
\frac{n+1}{2n}<\frac{1}{p}\leq 1,\qquad
0\leq \frac{1}{q}<\frac{n-1}{2n},\qquad
\frac{2}{n+1}\leq \frac{1}{p}-\frac{1}{q}\leq \frac{2}{n},
\end{equation}
$R_0(\zeta)$ is a bounded  linear operator from $L^p(\R^n)$ to $L^q(\R^n;\C)$ satisfying
\begin{equation}\label{eq:free-resolvent-est}
	\|R_0(\zeta)f\|_q
	\les |\zeta|^{\frac{n}{2}(\frac{1}{p} - \frac{1}{q} - \frac{2}{n})}\|f\|_p.
\end{equation}
Moreover, there are bounded linear operators $R_0(\lambda\pm i0):L^p(\R^n)\to L^q(\R^n;\C)$ such that 
$R_0(\zeta)f\to R_0(\lambda\pm i0)f$  as $\zeta\to\lambda \pm i0$  for all
$f\in L^p(\R^n)$ and
\begin{equation}\label{eq:free-resolvent-est2}
	\|R_0(\lambda \pm i 0)f\|_q
	\les \lambda^{\frac{n}{2}(\frac{1}{p} - \frac{1}{q} - \frac{2}{n})}\|f\|_p 
	\qquad (\lambda>0).
\end{equation}
\end{thm}
\begin{proof}
  The estimate~\eqref{eq:free-resolvent-est} is available in the literature mentioned above. The existence of
  a bounded linear operator $R_0(\lambda\pm i0)$ with
  $R_0(\zeta)f\wto R_0(\lambda\pm i0)f$ as $\zeta\to\lambda\pm i0$  follows from
  the uniform boundedness of the functions $R_0(\zeta)f$ in $L^q(\R^n)$ for $\zeta$
  near $\lambda$ (see \eqref{eq:free-resolvent-est}) and 
  the continuity of Cauchy type integrals as in~\cite[Theorem~4.1]{Agmon_Spectral}. We
  indicate how to prove that this convergence in fact holds in the strong sense.
  By density of test functions and~\eqref{eq:free-resolvent-est} it suffices to
  prove $R_0(\zeta)f-R_0(\tilde\zeta)f\to 0$ for test functions $f\in C_c^\infty(\R^n)$ 
  as $\zeta,\tilde\zeta\to \zeta_0 \in\C\sm\{0\}, \Imag(\zeta)\Imag(\tilde\zeta)>0$.  
  For simplicity we only consider the case $\Imag(\zeta),\Imag(\tilde\zeta)>0.$
  Here we can use $R_0(\zeta)f-R_0(\tilde\zeta)f = (G_\zeta-G_{\tilde\zeta}) \ast f$ where, according to
  \cite[p.46]{HuangYaoZheng}, we have for $\zeta=\mu^2\neq 0,\Real(\mu),\Imag(\mu)>0$ and 
  $\tilde\zeta=\tilde\mu^2$ sufficiently close to $\zeta$ with $\Real(\tilde\mu),\Imag(\tilde\mu)>0$,
  \begin{align*}
    |G_\zeta(z)-G_{\tilde\zeta}(z)|
    \les  \begin{cases}
      |\mu-\tilde\mu| |z|^{3-n} &,\text{if }|z|\leq |\mu|^{-1}\\
      |\mu-\tilde\mu| |\mu|^{\frac{n-3}{2}} |z|^{\frac{3-n}{2}} &,\text{if }
      |\mu|^{-1}\leq |z|\leq |\mu-\tilde\mu|^{-1}\\
      |\mu|^{\frac{n-3}{2}}  |z|^{\frac{1-n}{2}} &,\text{if }|z|\geq |\mu-\tilde\mu|^{-1}. 
    \end{cases}
  \end{align*}
  So Young's convolution inequality implies in view of $q>\frac{2n}{n-1}$
  \begin{align*}
    \|R_0(\zeta)f-R_0(\tilde\zeta)f\|_q 
    &\les   |\mu-\tilde\mu|\, \||z|^{3-n} \chi_{|z|\leq |\mu|^{-1}}\|_1 \|f\|_q \\
    & + |\mu-\tilde\mu|\, \||z|^{\frac{3-n}{2}}
     \chi_{|\mu|^{-1}\leq |z|\leq |\mu-\tilde\mu|^{-1}}    \|_q \|f\|_1 \\
    &+ \| |z|^{\frac{1-n}{2}} \chi_{|z|\geq |\mu-\tilde\mu|^{-1}}\|_q \|f\|_1 \\
    &\les   |\mu-\tilde\mu|  
     + |\mu-\tilde\mu|  \cdot |\mu-\tilde\mu|^{\frac{n-3}{2}-\frac{n}{q}}
     +  |\mu-\tilde\mu|^{\frac{n-1}{2}-\frac{n}{q}} \\ 
    &=   |\mu-\tilde\mu| +  |\mu-\tilde\mu|^{\frac{n-1}{2}-\frac{n}{q}}.
  \end{align*}
  Hence, $(R_0(\zeta)f)$ is a Cauchy sequence in $L^q$ and thus converges. Since the limit must coincide
  with the weak limit, we get the conclusion.
\end{proof}

The conditions~\eqref{eq:indices-Gutierrez} on $(p,q)$ are optimal for the uniform
estimates~\eqref{eq:free-resolvent-est2}, \textit{cf}.~\cite[p.1419]{KwonLee_Sharp}.  For any fixed
$\zeta\in \C \setminus \R_{\geq 0}$, however, the estimate~\eqref{eq:free-resolvent-est} actually holds for a
larger range of exponents, which is due to the improved properties of the Fourier symbol $1/(|\xi|^2-\zeta)$ 
and related Bessel potential estimates. We refer to~\cite{KwonLee_Sharp} for more details about sharp
$L^p-L^q$ resolvent estimates of the form~\eqref{eq:free-resolvent-est}. \\
Theorem~\ref{thm:Gutierrez} extends in an obvious way to the system case that we shall need in the following.  

\begin{cor}[Step 1]\label{cor:res-free-matrix}
	Let $m,n\in \N, n\geq 3$ and assume $\zeta \in \C\setminus \R_{\geq 0}$.
	Then, for $1\leq p,q\leq \infty$ as in~\eqref{eq:indices-Gutierrez}, $R_0(\zeta)$ is a bounded  linear
	operator from $L^p(\R^n;\C^m)$ to $L^q(\R^n;\C^m)$ satisfying
  $$
   \|R_0(\zeta)  f\|_q
     \les |\zeta|^{\frac{n}{2}(\frac{1}{p} - \frac{1}{q} - \frac{2}{n})} \|f\|_p.
  $$
 Moreover, there are bounded linear operators $R_0(\lambda\pm i0):L^p(\R^n;\C^m)\to L^q(\R^n;\C^m)$ such
 that $R_0(\zeta)f\to R_0(\lambda\pm i0)f$ as $\zeta\to\lambda\pm i0$  for all
 $f\in L^p(\R^n;\C^m)$. Furthermore, \eqref{eq:free-resolvent-est2} holds.
\end{cor}

\medskip

\subsection{Compactness of $K(\zeta)$}

We first proceed in greater generality by proving the boundedness and compactness of $K(\zeta)$ as an operator
from $L^{q_1}(\R^n;\C^m)$ to $L^{q_2}(\R^n;\C^m)$ for suitable $q_1,q_2$, as we will use this more general
result later. The proof of Step 2 then follows from the particular choice $q_1=q_2=q$, see
Corollary~\ref{cor:step2} below. In order to simplify the notation  in the proofs, 
we will write $L^s:=L^s(\R^n;\C^m),L^s_{\loc}:=L^s_{\loc}(\R^n;\C^m)$, etc.

\begin{prop}\label{prop:step2-gen}
	Let $n, m\in \N, n\geq 3$ and suppose that $V\in L^{[\kappa,\tilde\kappa]}(\R^n;\C^{m\times m})$
	where $1\leq \kappa\leq \tilde\kappa<\infty$.
	Then, for $\zeta \in \C\setminus \R_{\geq 0}$ or $\zeta=\lambda \pm i 0, \lambda>0$, the operator
	$K(\zeta)=-R_0(\zeta)V\colon L^{q_1}(\R^n;\C^m) \to L^{q_2}(\R^n;\C^m)$ is compact  provided that the
	following conditions hold for $q_1,q_2\in [1,\infty]$:
	\begin{equation}\label{eq:ass-compactness}
		\frac{n+1}{2n} - \frac{1}{\tilde\kappa}<\frac{1}{q_1}\leq 1-\frac{1}{\kappa},\qquad
		0\leq \frac{1}{q_2}<\frac{n-1}{2n},\qquad
		\frac{2}{n+1}-\frac{1}{\tilde\kappa}\leq \frac{1}{q_1} -\frac{1}{q_2}\leq \frac{2}{n}-\frac{1}{\kappa}.
	\end{equation} 
	Moreover, 
	\begin{equation}\label{eq:normestimate_K}
	  \|K(\zeta)\|_{q_1\to q_2}
	  \les \inf_{V=V_1+V_2} 
	  \left[ |\zeta|^{\frac{n}{2}(\frac{1}{q_1}-\frac{1}{q_2}+\frac{1}{\kappa}-\frac{2}{n})} 
	   \|V_1\|_\kappa + |\zeta|^{\frac{n}{2}(\frac{1}{q_1}-\frac{1}{q_2}+\frac{1}{\tilde\kappa}-\frac{2}{n})}
	   \|V_2\|_{\tilde\kappa}\right].
	\end{equation}
	Furthermore, $u_j\wto u,\zeta_j\to \zeta$ implies $K(\zeta_j)u_j\to K(\zeta)u$.  
	In particular, the operators $K(\zeta)$ depend continuously on $\zeta\in\C\sm\R_{\geq 0}$ in the uniform
	operator topology and we have $K(\zeta)\to K(\lambda\pm i0)$ as $\zeta\to\lambda\pm i0$.
\end{prop}
\begin{proof}
	We begin with proving  \textit{boundedness} of $K(\zeta)\colon L^{q_1}  \to L^{q_2}$. 
	Using the decomposition $V=V_1+V_2,$ $V_1\in L^{\kappa}(\R^n;\C^{m\times m})$ and $V_2\in
	L^{\tilde\kappa}(\R^n;\C^{m\times m}),$ it follows from Corollary~\ref{cor:res-free-matrix} that
	\begin{equation*}
	  \|K(\zeta) f\|_{q_2}
	   \leq \|R_0(\zeta)\|_{p\to q_2} \|V_1 f\|_{p} + 
	   \|R_0(\zeta)\|_{\tilde p\to q_2}\|V_2 f\|_{\tilde p}
	\end{equation*}
	whenever the tuples $(p,q_2)$ and $(\widetilde{p}, q_2)$ satisfy the conditions
	in~\eqref{eq:indices-Gutierrez}. In view of~\eqref{eq:ass-compactness} and $\kappa\leq \tilde\kappa$ these
	conditions are satisfied if we choose $p,\tilde p$ according to $\frac{1}{p}=\frac{1}{\kappa} +
	\frac{1}{q_1}$, $\frac{1}{\tilde p}=\frac{1}{\tilde\kappa} + \frac{1}{q_1}$. So  
	H\"older's inequality and Corollary~\ref{cor:res-free-matrix} give
	\begin{equation*}
	  \|K(\zeta) f\|_{q_2}
	   \les \left(|\zeta|^{\frac{n}{2}(\frac{1}{q_1}-\frac{1}{q_2}+\frac{1}{\kappa}-\frac{2}{n})} 
	   \|V_1\|_\kappa + |\zeta|^{\frac{n}{2}(\frac{1}{q_1}-\frac{1}{q_2}+\frac{1}{\tilde\kappa}-\frac{2}{n})}
	   \|V_2\|_{\tilde\kappa}\right) \|f\|_{q_1},
	\end{equation*}
	which proves the claimed boundedness as well as~\eqref{eq:normestimate_K}.
	
	\medskip  
		
	Next we show that $u_j\wto u,\zeta_j\to \zeta$ implies $K(\zeta_j)u_j\to K(\zeta)u$.
	Here, $\zeta\in\C\sm\R_{\geq 0}$ or $\zeta=\lambda\pm i0,\lambda>0$. Notice that 
	this fact and Corollary~\ref{cor:res-free-matrix} imply the compactness of $K(\zeta)$ (choose
	$\zeta_j:=\zeta$) as well as the  existence of a continuous extension of $\zeta\mapsto K(\zeta)$ in
	$\mathcal B(L^{q_1};L^{q_2})$ to the closed upper resp. lower complex half-planes.

As shown in~\cite[Lemma~3.1]{GolSch_LAP}, without loss of generality, we can assume $V$ to be  
	bounded and to have compact support. 
	
	\medskip
	
	We first prove $K(\zeta_j)u_j\to K(\zeta)u$ in $L^{q_2}_{\loc}$. To this end, it suffices to prove
	the uniform boundedness of the operators $K(\zeta_j)\colon L^{q_1}\to W^{2,q_1}(B;\C^m)$ with respect to
	$j$ for any given bounded ball $B\subset\R^n$. 
	Indeed, the embedding $W^{2,q_1}(B;\C^m)\hookrightarrow L^{q_2}(B;\C^m)$ is compact due to
	$\frac{1}{q_1}-\frac{1}{q_2}\leq \frac{2}{n}-\frac{1}{\kappa}<\frac{2}{n}$ and the Rellich-Kondrachov
	Theorem, which implies $K(\zeta_j)u_j\to v$ in $L^{q_2}(B;\C^m)$ for some $v$. This and $u_j\wto u$
	implies $v=K(\zeta)u.$ 
So we conclude $K(\zeta_j)u_j\to K(\zeta)u$ in $L^{q_2}_{\loc}$ once
	we have proved the uniform boundedness of $K(\zeta_j)\colon L^{q_1}\to W^{2,q_1}(B;\C^m)$.  
	
	\medskip
	
	To prove this let $f\in	L^{q_1}$ be arbitrary. Then $w_j:=K(\zeta_j)f$ satisfies the elliptic system
	\begin{equation*}
		\Delta w_j= (\Delta + \zeta_j)w_j - \zeta_j w_j= -Vf -\zeta_j w_j \quad\text{in }2B. 		 
	\end{equation*}
	From elliptic interior regularity estimates and  the mapping properties of $K(\zeta_j)$ stated in
	Corollary~\ref{cor:res-free-matrix} and H\"older's inequality we obtain $ \|w_j\|_{W^{2,q_1}(B;\C^m)}\les  \|f\|_{q_1},$
	which is what we had to prove. 
	\medskip
	
	To conclude it is sufficient to show $\sup_j \|\chi_{\R^n\setminus B} K(\zeta_j)\|_{q_1\to q_2}
	\to 0$ as $B\nearrow \R^n.$ To  this end we use  
	\begin{equation*}
		K(\zeta_j) f(x)= -\int_{\R^n} G_{\zeta_j}(x-y) V(y)f(y)\, \mathrm{d}y,
	\end{equation*}
   where $G_{\zeta_j}(z)$ is the integral kernel of the resolvent operator $R_0(\zeta_j)$, which is explicitly given
   in terms of Bessel functions. We use the bound $\sup_j
   |G_{\zeta_j}(z)|\les |z|^\frac{1-n}{2}$ for $|z|\geq 1$, see  (2.21),(2.25) in~\cite{K_R_S}.
    Recalling that $V$ is assumed to be bounded and compactly supported  we infer for $M:=\supp V$ 
\begin{equation*} 
	|K(\zeta_j)f(x)|\les \int_M |x-y|^\frac{1-n}{2}|V(y)| |f(y)|\, \mathrm{d}y
	\les |x|^\frac{1-n}{2} \|V\|_{q_1'} \|f\|_{q_1}\quad\text{if }\dist(x,M)\geq 1.
\end{equation*}
  This yields for large enough balls $B$
\begin{equation*}
	\|\chi_{\R^n\setminus B} K(\zeta_j) f\|_{q_2}
	\les \|V\|_{q_1'} \|f\|_{q_1} \Big(\int_{\R^n \setminus B} |x|^{\frac{q_2(1-n)}{2}}\,
	\mathrm{d}x\Big)^\frac{1}{q_2}
\end{equation*}
and the conclusion follows due to $q_2>\frac{2n}{n-1}$.   
\end{proof}

The second step now results from considering the special case $q_1=q_2=q$
in Proposition~\ref{prop:step2-gen}.

\begin{cor}[Step 2]\label{cor:step2}
	Let $n, m\in \N,$ $n\geq 3$ and assume that $V\in L^{[\kappa,\tilde\kappa]}(\R^n;\C^{m\times m})$
	where $\frac{n}{2}\leq \kappa\leq \tilde\kappa\leq \frac{n+1}{2}$.
	Then, for $\zeta \in \C\setminus \R_{\geq 0}$
	or $\zeta=\lambda\pm i0, \lambda>0$, the operator $K(\zeta)=-R_0(\zeta) V\colon L^q(\R^n;\C^m)\to
	 L^q(\R^n;\C^m)$ is compact provided that
	\begin{equation}\label{eq:simplified-cond}
		\frac{n+1}{2n}-\frac{1}{\tilde\kappa}<\frac{1}{q}<\frac{n-1}{2n}.
	\end{equation}
	Furthermore, $u_j\wto u,\zeta_j\to \zeta$ implies $K(\zeta_j)u_j\to K(\zeta)u$.  
	In particular, the operators $K(\zeta)$ depend continuously on $\zeta\in\C\sm\R_{\geq 0}$ in the uniform
	operator topology and we have $K(\zeta)\to K(\lambda\pm i0)$ as $\zeta\to\lambda\pm i0$.
\end{cor}

\medskip
\subsection{Injectivity of   $I-K(\zeta)$}

We now prove the injectivity of the Fredholm operator $I-K(\zeta)\colon
L^q(\R^n;\C^m)\to L^q(\R^n;\C^m)$ for $q$ as in~\eqref{eq:simplified-cond}. So we have to show that 
\begin{equation}\label{eq:ev-eq} 
	u-K(\zeta)u=0, \qquad u\in L^q(\R^n;\C^m)
\end{equation}
implies $u=0$. As a starting point, using a   bootstrapping procedure, we show that solutions
of~\eqref{eq:ev-eq} display both more local integrability and better decay at infinity.  

\begin{prop}\label{prop:regularity}
Let $n,m\in \N,$ $n\geq 3$,  $q$ as in~\eqref{eq:simplified-cond} and assume $V\in
L^{[\kappa,\tilde\kappa]}(\R^n;\C^{m\times m})$ where  $\frac{n}{2}\leq \kappa\leq \tilde\kappa\leq \frac{n+1}{2}.$
Then any solution $u\in L^q(\R^n;\C^m)$ of~\eqref{eq:ev-eq} belongs to  
\begin{align*}
  L^r(\R^n;\C^m)&\cap H^1_{\loc}(\R^n;\C^m)
  \text{ for all }r\in \Big(\frac{2n}{n-1},\frac{2n}{n-3}\Big)
  \,\quad\text{when } \kappa=\frac{n}{2},  \\
  L^r(\R^n;\C^m)&\cap H^1_{\loc}(\R^n;\C^m) 
  \text{ for all } r\in \Big(\frac{2n}{n-1},\infty\Big] 
  \quad\qquad\text{when } \kappa>\frac{n}{2}.
\end{align*}
  Moreover, for any given such $r,q$ we have $\|u\|_r\les  \|u\|_q$. 
\end{prop}
\begin{proof}
  We write again $L^s:=L^s(\R^n;\C^m)$. As a starting point we
  show that any given solution $u\in L^q$ of~\eqref{eq:ev-eq} belongs to $L^r$ for all $r\in (\frac{2n}{n-1},q),$ in other words $u$ displays a 
  \textit{better decay at infinity}.
  We give a proof of this fact distinguishing between $\tilde\kappa<\frac{n+1}{2}$ and the limiting case
  $\tilde\kappa=\frac{n+1}{2}.$ Let us first consider $\tilde\kappa<\frac{n+1}{2}$. Define
	$\frac{1}{q_0}<\frac{1}{q_1}<\dots<\frac{1}{q_j}<\frac{1}{q_{j+1}}<\dots<\frac{n-1}{2n}$ by
	\begin{equation}\label{eq:strictly_less}
	  \frac{1}{q_0} := \frac{1}{q},\qquad 
		\frac{1}{q_{j+1}}:=\min \bigg\{\frac{1}{2}\Big(\frac{1}{q_j} + \frac{n-1}{2n}\Big),
		\frac{1}{q_j}-\frac{2}{n+1} + \frac{1}{\tilde\kappa} \bigg\} \qquad (j\in \N_0).
	\end{equation}
	Since we are assuming $\tilde{\kappa}<\frac{n+1}{2},$ at each iteration we indeed get a smaller Lebesgue 
	exponent, namely $\frac{n-1}{2n}>\frac{1}{q_{j+1}}>\frac{1}{q_j}$. 
	Then one shows that the tuple $(q_j,q_{j+1})$
	satisfies the conditions~\eqref{eq:ass-compactness} in Proposition~\ref{prop:step2-gen}, thus $K(\zeta)$ maps $L^{q_j}$ to
	$L^{q_{j+1}}.$ 

	Applying iteratively Proposition~\ref{prop:step2-gen} to the equation~\eqref{eq:ev-eq} we obtain $u\in L^{q_j}$ for all $j\in \N_0.$ From
	$\frac{1}{q_j}\nearrow \frac{n-1}{2n}$ we infer $u\in L^r$ for all $\frac{2n}{n-1}<r< q$ by interpolation as
	well as $\|u\|_r\les  \|u\|_q$.
   
    \medskip
    
  Now we consider the limiting case $\tilde\kappa=\frac{n+1}{2}$. Observe that, according to the previous
  definition~\eqref{eq:strictly_less} of  $q_{j+1},$ in this limiting situation there
  would be no decay gain at each iteration because of $q_{j+1}=q_j$. We can circumvent this by choosing a 
  decomposition  $V=V_1+V_2,$ where $V_1\in L^{\frac{n}{2}}(\R^n;\C^{m\times m})$ and $V_2\in
  L^\frac{n+1}{2}(\R^n;\C^{m\times m})$ has a small $L^\frac{n+1}{2}-$norm.

  We use this observation in order to justify a similar iteration as above, this time for exponents
  $\frac{1}{q_0}<\frac{1}{q_1}<\dots<\frac{1}{q_j}<\frac{1}{q_{j+1}}<\dots<\frac{n-1}{2n}$ given by 
  $$
  \frac{1}{q_0} := \frac{1}{q},\qquad \frac{1}{q_{j+1}}:=\min \bigg\{\frac{1}{2}\Big(\frac{1}{q_j} +
  \frac{n-1}{2n}\Big), \frac{1}{q_j}-\frac{2}{n+1} + \frac{2}{n} \bigg\} \qquad (j\in \N_0).
  $$
  We have to show that $u=K(\zeta)u,u\in L^{q_j}$ implies $u\in L^{q_{j+1}}$. Having done this,
  we conclude from  $\frac{1}{q_j}\nearrow\frac{n-1}{2n}$ and interpolation that $u\in L^r$ for all
  $\frac{2n}{n-1}<r<q$ as well as $\|u\|_r \les \|u\|_q$, which then finishes  
  the proof of our first claim.
  
  \medskip
  
  So  assume $u\in L^{q_j}$. One can choose $V=V_1 + V_2$ such that
  \begin{equation}\label{eq:smallness_V2}
	C_j \|V_2\|_{\frac{n+1}{2}}<\frac{1}{2}
  \end{equation}  
  where $C_j$ denotes the operator norm of $R_0(\zeta)\colon L^{s_j}  \to L^{q_{j+1}}$ where $s_j$ is defined
  via $\frac{1}{s_j}-\frac{1}{q_{j+1}}=\frac{2}{n+1}$.   
  Observe that this operator norm is finite due to Corollary~\ref{cor:res-free-matrix} because of
  $\frac{n+1}{2n}<\frac{1}{s_j}\leq 1.$

  We introduce the auxiliary operator $T:=I+R_0(\zeta) V_2: L^{q_{j+1}}\to L^{q_{j+1}}$. 
Using~\eqref{eq:smallness_V2} we find that $T$ is bounded and invertible due to 
Corollary~\ref{cor:Helmholtz_Bijectivity}. 
  So $T$ has a bounded inverse $T^{-1}:L^{q_{j+1}}\to L^{q_{j+1}}$.  
 Since $u$ satisfies~\eqref{eq:ev-eq}, we have
$$
		\int_{\R^n} u\phi\, \mathrm{d}x=\int_{\R^n} K(\zeta)u\phi\, \mathrm{d}x
		=-\int_{\R^n} R_0(\zeta)V_1 u\phi\, \mathrm{d}x
		- \int_{\R^n} R_0(\zeta)V_2 u\phi\, \mathrm{d}x
$$
 for any given $\phi\in C_c^\infty(\R^n;\C^m)$. 
Thus one has
\begin{equation}\label{eq:dual_T_Estimate}
	\Big| \int_{\R^n} Tu\phi\, \mathrm{d}x \Big|
	\leq\Big |\int_{\R^n} R_0(\zeta) V_1 u \phi \,\mathrm{d}x \Big|
	\leq \|R_0(\zeta)V_1 u\|_{q_{j+1}}\|\phi\|_{q_{j+1}'}
	\les  \|u\|_{q_j} \|\phi\|_{q_{j+1}'}
\end{equation}	
by Proposition~\ref{prop:step2-gen}.
By the the dual characterization of the Lebesgue norms $\|\cdot\|_{q_{j+1}}$ and using
density of the test functions in $L^{q_{j+1}'},$ one gets 
$\|Tu\|_{q_{j+1}}\leq \|u\|_{q_j}$. Then the
boundedness of $T^{-1}$ gives $\|u\|_{q_{j+1}}\les \|u\|_{q_j}$, which is what we had to prove.

\medskip

Next we prove \textit{higher integrability} of $u$. In the case $\kappa>\frac{n}{2}$ classical
Moser iteration implies $u\in L^\infty$ and the claim follows by interpolation. So it remains to prove $u\in
L^r$ for $r\in (q,\frac{2n}{n-3})$ in the limiting case $\kappa=\frac{n}{2}$. This can be proved similarly
to the limiting case $\tilde{\kappa}=\frac{n+1}{2}$ above using the decomposition $V=V_1 + V_2$ with $V_1\in
L^{\frac{n+1}{2}}(\R^n;\C^{m\times m})$ and $V_2\in L^{\frac{n}{2}}(\R^n;\C^{m\times m})$ which has a small
norm $L^{\frac{n}{2}}-$ norm.    For the sake of brevity we omit the details.

 \medskip 
 
  Finally, we discuss the $H^1_{\loc}$-regularity. In the case $\kappa=\frac{n}{2}$ solutions $u$
  of~\eqref{eq:ev-eq} satisfy $\Delta u = -(\zeta+V)u$ in~$\R^n$ in the distributional sense.  
 Since $\zeta+V\in L^{\frac{n}{2}}_{\loc}(\R^n;\C^{m\times m})$ and $u\in L^r_{\loc}$ for all $r<\frac{2n}{n-3}$, we
 conclude $\Delta u \in L^s_{\loc}$ for all $s\in [1,\frac{2n}{n+1})$. So Cald\'{e}ron-Zygmund
 estimates yield  $u\in W^{2,s}_{\loc}$ for those exponents and thus $u\in H^1_{\loc}$ by Sobolev's Embedding
 Theorem, which finishes the proof for $\kappa=\frac{n}{2}$. In the case $\kappa>\frac{n}{2}$ 
 one obtains similarly $u\in W^{2,\kappa}_{\loc}\subset H^1_{\loc}$ and the proof is
 finished.
\end{proof}

With these improved integrability and regularity properties of solutions at hand, we may turn towards the
injectivity of $I-K(\zeta)$. As in the paper by Goldberg and Schlag~\cite{GolSch_LAP} we will discuss
separately the case $\zeta \in \C\setminus \R_{\geq 0}$ and the  more involved situation 
$\zeta=\lambda\pm i0,\lambda>0$. Observe that if $u$ is a solution to~\eqref{eq:ev-eq}, then it solves 
the corresponding eigenvalue equation   
\begin{equation}\label{eq:ev-pb}
	(\Delta + \zeta)u + Vu=0,\qquad u\in L^q(\R^n;\C^m).
\end{equation}
To prove injectivity of $I-K(\zeta)$ one is thus lead to prove the absence of 
embedded eigenvalues for the Schr\"odinger operator $\Delta + V$. 
However, for $\zeta=\lambda>0$, the Helmholtz equation~\eqref{eq:ev-pb} may possess nontrivial solutions:
consider for instance $V=0$ and ordinary Herglotz waves $u$ with pointwise decay rate $|u(x)|\les
(1+|x|)^{(1-n)/2}$.
As a consequence, some extra decay condition at infinity coming from the integral representation of $u$ from~\eqref{eq:ev-eq} has to be
used in order to deduce $u\equiv 0$.
In~\cite{GolSch_LAP} the authors managed to prove that in the case $n=3,m=1$ solutions $u\in
L^4(\R^3;\C)$ of~\eqref{eq:ev-eq} even satisfy $(1+|\cdot|)^{\tau_1-\frac{1}{2}}u\in
L^2(\R^3;\C)$ for some $\tau_1>0$ so that a fundamental result by
Ionescu-Jerison~\cite{IonJer_OnTheAbsence} on the absence of embedded eigenvalues  allows to conclude $u\equiv 0$.
The ideas presented in~\cite{GolSch_LAP} are not limited to $n=3$, but carry
over to general dimensions $n\geq 2$ in a straightforward manner. In order to avoid redundancy we only state
the (scalar) results that generalize Lemma~2.3 and Proposition~2.4 from this paper.

  \begin{prop}[Goldberg-Schlag]  \label{prop:trace}
    Let $n\in\N,n\geq 2$ and assume $f\in L^p(\R^n)$ where     
    $1\leq p\leq \frac{2(n+1)}{n+3}$. Then we have for  $|t|<\frac{1}{2}$ and 
    $\gamma:= \frac{1}{2}\min\{1,\frac{n+1}{p}-\frac{n+3}{2}\}$ the
    estimate 
    $$ 
      \|\hat f((1+t)\cdot)\|_{L^2(S_\lambda)} 
      \les  |t|^{\gamma} \|f\|_p 
    $$
    provided that $\hat f$ vanishes identically on the unit sphere in $\R^n$. 
  \end{prop}
     
  \begin{prop}[Goldberg-Schlag]  \label{prop:decay}
    Let $n\in\N,n\geq 2$ and $f\in L^p(\R^n)$ for 
    $\max\{1,\frac{2n}{n+4}\}\leq  p\leq   \frac{2(n+1)}{n+3},\, (n,p)\neq (4,1)$.  
    Then we have for all $\tau_1 < \frac{1}{2}\min\{1,\frac{n+1}{p}-\frac{n+3}{2}\}$ 
    the estimate 
    $$
      \|(1+|\cdot|)^{\tau_1-\frac{1}{2}} R_0(\lambda\pm i0)  f\|_2 
      \les \|f \|_p 
    $$
    provided that $\hat f$ vanishes identically on the unit sphere in $\R^n$.
  \end{prop}

  Clearly, Proposition~\ref{prop:decay} generalizes to systems simply by considering each component
  separately. This is how we will deduce $(1+|\cdot|)^{\tau_1-\frac{1}{2}}u\in L^2(\R^n;\C^m)$
  for some $\tau_1>0$, which is crucial for the absence of embedded eigenvalues stated in
  Theorem~\ref{thm:abs-emb-evs} below.

To some extent this result is already contained in the work~\cite[Theorem~3]{KochTat_Absence} by Koch and
Tataru. Even so, for the reader's convenience, we decided to provide a the proof.
 
\begin{thm}\label{thm:abs-emb-evs}
	Let $n,m\in \N, n\geq 3$, assume $V\in L^{[\frac{n}{2},\frac{n+1}{2}]}(\R^n;\C^{m\times m})$ and
	let $u\in H^1_{\loc}(\R^n;\C^m)$ be a solution of~\eqref{eq:ev-pb} for $\zeta\in\R_{>0}$ satisfying
	$|x|^{\tau_1-\frac{1}{2}} u \in L^2(\R^n;\C^m)$. Then $u\equiv 0$.  
\end{thm}
\begin{proof}
The key point is the $L^p$-Carleman estimate proved in~\cite{KochTat_Absence}. More precisely, introducing
the Carleman weight $h_\epsilon(t)$ such that $ h_\epsilon'(t)=\tau_1 + (\tau e^\frac{t}{2}-
\tau_1)\frac{\tau^2}{\tau^2 + \epsilon e^t},$ $\epsilon >0,$
Koch and Tataru proved in~\cite[Proposition~4]{KochTat_Absence}  
	\begin{equation}\label{eq:Carleman}
		\|e^{h_\epsilon(\ln|x|)}v\|_{\frac{2n}{n-2} }
		+ \|e^{h_\epsilon(\ln|x|)}v\|_{ \frac{2(n+1)}{n-1}}
		\lesssim
		\inf_{(\Delta + \zeta)v=f +g}
		\|e^{h_\epsilon(\ln|x|)}f\|_{\frac{2n}{n+2}}
		+ \|e^{h_\epsilon(\ln|x|)}g\|_{\frac{2(n+1)}{n+3}}
	\end{equation}
	for all $v$ supported in $\R^n\sm B_1$ such that $|x|^{\tau_1 - \frac{1}{2}}v \in
	L^2(\R^n)$. Notice that~\eqref{eq:Carleman} is uniform with respect to
	$0<\epsilon\leq \epsilon_0$ and $\tau\geq \tau_0>0$ for some $\eps_0,\tau_0>0.$
	Notice that in~\cite{KochTat_Absence} the estimate~\eqref{eq:Carleman} is proved even for more
	general classes of Helmholtz-type operators also allowing for long-range perturbations. Moreover, the
	corresponding estimate~(7) in that paper is formulated with even stronger norms on the left and weaker
	norms on the right, as one can check using the embeddings
	\begin{align} \label{eq:Embeddings}
 \begin{aligned}
  L^{\frac{2(n+1)}{n+3}}(\R^n) +  L^{\frac{2n}{n+2}}(\R^n)
  &\hookrightarrow W^{-\frac{1}{n+1},\frac{2(n+1)}{n+3}}(\R^n),\\
  L^{\frac{2(n+1)}{n-1}}(\R^n) \cap  L^{\frac{2n}{n-2}}(\R^n)
  &\hookleftarrow  W^{\frac{1}{n+1},\frac{2(n+1)}{n-1}}(\R^n).
  \end{aligned}
	\end{align}

	\medskip

	\noindent \textit{Step 3.1: Exponential decay.}\;   As anticipated  we first show that $u$ decays at infinity
	faster than $e^{-\tau|x|^{1/2}}$ in some integrated sense. In order to apply~\eqref{eq:Carleman} for this purpose, we
	need to localize the support of $u$ to a spatial region far from the origin. So we pick a non-negative bump
	function $\phi\in C^\infty(\R^n)$ such that $\phi(x)=1$ for $|x|>2R$ and $\phi(x)=0$ for $|x|\leq R$ for a
	sufficiently large $R$ to be chosen later and define $v:=\phi u$. Then $v$
	solves
	$$
	  (\Delta+\zeta)v = \left[(\Delta\phi) u + 2\nabla \phi \cdot \nabla u + V_1 v\right] + V_2 v.
	$$
	Now we apply the scalar Carleman estimate~\eqref{eq:Carleman} to each component of the system.
	Summing up the resulting estimates one gets
	\begin{align*}
	  \begin{aligned}
		&\|e^{h_\epsilon(\ln|x|)} v\|_{\frac{2n}{n-2}}	+
			\| e^{h_\epsilon(\ln|x|)} v\|_{\frac{2(n+1)}{n-1}} \\
		&\lesssim
			  \|e^{h_\epsilon(\ln|x|)}\left((\Delta\phi) u + 2\nabla \phi \cdot \nabla u + V_1
			v\right) \|_{\frac{2n}{n+2}} +	\|e^{h_\epsilon(\ln|x|)}V_2v\|_{\frac{2(n+1)}{n+3}} \\
	   &\les_R
	    \|e^{h_\epsilon(\ln|x|)}\left(|u|+|\nabla u|\right)\chi_{B_{2R}\sm B_R} \|_{\frac{2n}{n+2}} \\
	     &+  \|e^{h_\epsilon(\ln|x|)}V_1v\|_{\frac{2n}{n+2}}
	     + \|e^{h_\epsilon(\ln|x|)}V_2v\|_{\frac{2(n+1)}{n+3}} \\
	   &\les  \|e^{h_\epsilon(\ln|x|)}\left(|u|+|\nabla u|\right)\chi_{B_{2R}\sm B_R} \|_{\frac{2n}{n+2}} \\
	   & + \|V_1 \chi_{\R^n\sm B_R}\|_{\frac{n}{2}} \|e^{h_\epsilon(\ln |x|)}     v\|_{\frac{2n}{n-2}}
	     +  \|V_2 \chi_{\R^n\sm B_R}\|_{\frac{n+1}{2}} \|e^{h_\epsilon(\ln |x|)}     v\|_{\frac{2(n+1)}{n-1}}.
   \end{aligned}
 \end{align*}
 Here we used that $\phi$ is supported in $\R^n\setminus B_R$ and $|\nabla\phi|,\Delta\phi$ are supported in
 $\ov{B_{2R}}\setminus B_R$.  For $R$  sufficiently large we can absorb the potential-dependent terms on the
 right-hand side and get $$
	\|e^{h_\epsilon(\ln|x|)} v\|_{\frac{2n}{n-2}}
	+
	\|e^{h_\epsilon(\ln|x|)} v\|_{\frac{2(n+1)}{n-1}}
	\lesssim_R \|e^{h_\epsilon(\ln|x|)}   (|u|+ |\nabla u|)\chi_{B_{2R}\sm B_R}\|_{ \frac{2n}{n+2}}.
 $$
 Since $h_\epsilon(\ln|x|)\nearrow \tau |x|^{1/2}$ as $\epsilon\to 0^+$, the Monotone Convergence
 Theorem implies
\begin{equation}\label{eq:preliminary3}
	\|e^{\tau|x|^{1/2}} v\|_{\frac{2n}{n-2}} +	\|e^{\tau|x|^{1/2}} v\|_{\frac{2(n+1)}{n-1}}
	\lesssim_R  \|e^{\tau|x|^{1/2}}   (|u|+ |\nabla u|)\chi_{B_{2R}\sm B_R} \|_{\frac{2n}{n+2}}.
\end{equation}
Since the right side is finite by Proposition~\ref{prop:regularity} (notice that the presence of the
exponential factor is irrelevant as we are localized in a bounded region), $v$ and hence   $u$ have
exponential decay at infinity.

\medskip

\noindent \textit{Step 3.2: Compact support.}\; From~\eqref{eq:preliminary3} we even infer that $u$ is
supported in $\ov{B_{2R}}$.  Indeed, $|x|^{1/2}\leq \sqrt{2R}$ on $B_{2R}\sm B_R$ implies
\begin{equation*}
	e^{-\tau \sqrt{2R}}
	\Big(
	\|e^{\tau |x|^{1/2}} v\|_{\frac{2n}{n-2}} + \|e^{\tau|x|^{1/2}} v\|_{\frac{2(n+1)}{n-1}}
	\Big)
	 \lesssim_R \| (|u|+ |\nabla u|)\chi_{B_{2R}\sm B_R} \|_{\frac{2n}{n+2}}.
\end{equation*}
Letting $\tau$ go to infinity shows that $v$ and hence $u$ vanishes identically outside $\ov{B_{2R}}$.

\medskip

\noindent \textit{Step 3.3: Triviality.}\;
In virtue of the conclusions provided by Step 3.2, proving the triviality of $u$ reduces to proving the weak
unique continuation property for the differential inequality $|\Delta u|\leq |V| |u|.$
At this stage we only need that our potential $V\in L^{[\frac{n}{2},\frac{n+1}{2}]}(\R^n;\C^{m\times m})$
belongs to $L^\frac{n}{2}_{\loc}(\R^n; \C^{m \times m}),$ as only local properties of the solution $u$ are
investigated. So~\cite[Theorem~6.3]{JerisonKenig} applies and we obtain $u\equiv 0$.
\end{proof}

Now we are in the position to prove the injectivity of $I-K(\zeta)$.

  \begin{cor}[Step 3] \label{cor:Helmholtz_Bijectivity}
	Let $n,m\in \N,$ $n\geq 3$ and assume $V=\ov{V}^T\in L^{[\kappa,\tilde\kappa]}(\R^n;\C^{m\times m})$ where
	$\frac{n}{2}\leq\kappa\leq \tilde\kappa\leq \frac{n+1}{2}$.
	Then, for $\zeta\in \C\setminus \R_{\geq 0}$ or $\zeta=\lambda \pm i 0,\lambda>0$, the operator $I
	-K(\zeta)\colon L^q(\R^n;\C^m)\to L^q(\R^n;\C^m)$ is bijective	
	and satisfies 
	$$
	  \|(I-K(\zeta))^{-1}\|_{q\to q} \leq C<\infty \qquad\text{for all }\zeta\in\C\sm\R_{\geq 0} 
	$$
	provided that $q$ satisfies~\eqref{eq:simplified-cond}.
	Moreover,   $\zeta\mapsto (I-K(\zeta))^{-1}$ is continuous on  $\C\sm\R_{\geq 0}$  and we have
	$(I-K(\zeta))^{-1}\to (I-K(\lambda\pm i0))^{-1}$  in the uniform operator topology as
	$\zeta\to\lambda,\Imag(\zeta)\to 0^\pm$.
\end{cor}
\begin{proof}
  Once again we write $L^s:=L^s(\R^n;\C^m)$. We  only consider the case $\zeta=\lambda +i0,$ $\lambda>0$
  since the remaining case $\zeta \in \C\setminus \R_{\geq 0}$ is straightforward and can be handled as
  in equation (3.6) in~\cite{GolSch_LAP}. We first prove that the operator is injective, so our aim is to show
  that any given solution $u$ of~\eqref{eq:ev-eq} must  be trivial.  
  Proposition~\ref{prop:regularity} implies $u\in L^r$
  whenever $\frac{n-3}{2n}<\frac{1}{r}<\frac{n-1}{2n}$. So $V\in
  L^{[\kappa,\tilde\kappa]}(\R^n;\C^{m\times m})\subset L^{[\frac{n}{2},\frac{n+1}{2}]}(\R^n;\C^{m\times m})$
  implies $Vu\in L^{[p_1,p_2]}(\R^n;\C^m)$ for some $p_1,p_2$ satisfying $$
	\frac{2}{n} + \frac{n-3}{2n}<\frac{1}{p_1}  <\frac{2}{n} + \frac{n-1}{2n},
	\qquad 
	\frac{2}{n+1}+ \frac{n-3}{2n}<\frac{1}{p_2}<\frac{2}{n+1} + \frac{n-1}{2n}.
 $$
 This implies  $Vu\in L^p$ where $\frac{2}{n} + \frac{n-3}{2n}<\frac{1}{p}  <  \frac{2}{n+1} +
 \frac{n-1}{2n}$. In particular, we deduce from 
  $u\in L^r$ for $\frac{n-3}{2n}<\frac{1}{r}<\frac{n-1}{2n}$ the statement 
 $$
   u \in L^{p'},\; g:=Vu\in L^p
   \qquad\text{whenever}\quad \frac{n+1}{2n}  <\frac{1}{p}  <  \frac{n^2+4n-1}{2n(n+1)}.
 $$
 So a density argument and $V=\ov{V}^T$ imply
\begin{align} \label{eq:equation}
  0
   = \Imag\left(\langle u, Vu \rangle\right) 
   \stackrel{\eqref{eq:ev-eq}}= \Imag\left(\langle K(\lambda+i0) u, g \rangle\right)
    =- \Imag\left(\langle R_0(\lambda+i0)g,g \rangle\right) 
  =- c\sqrt{\lambda}\int_{S^{n-1}} |\widehat{g}(\sqrt{\lambda}\omega)|^2 \mathrm{d}\sigma(\omega)
\end{align}
for some positive number $c>0$, \textit{cf}. (3.7) in~\cite{GolSch_LAP}. This implies that $\hat g$ vanishes
identically on the sphere of radius $\sqrt\lambda$ so that Proposition~\ref{prop:decay} 
(choose $\frac{2n(n+1)}{n^2+4n-1}<p<\frac{2(n+1)}{n+3}$) implies 
 $$
   (1+|\cdot|)^{\tau_1-\frac{1}{2}} u
   = (1+|\cdot|)^{\tau_1-\frac{1}{2}} K(\lambda+i0)u
   = -(1+|\cdot|)^{\tau_1-\frac{1}{2}} R_0(\lambda+i0)  g \in L^2 
 $$
 provided that $0<\tau_1<\frac{1}{2}\min\{1,\frac{n+1}{p}-\frac{n+3}{2}\}$. In view of 
 $u\in H^1_{\loc}$, which is a consequence of  Proposition~\ref{prop:regularity}, we see that the hypotheses
 of Theorem~\ref{thm:abs-emb-evs} are satisfied and we conclude $u\equiv 0$, which proves the injectivity of
 $I-K(\zeta)$ and hence its invertibility.
 
 \medskip
 
To prove the continuity of $\zeta\mapsto (I-K(\zeta))^{-1}$ from 
$\C\setminus \R_{\geq0}$ to the space $\mathcal B(L^q;L^q)$ we use the identity
\begin{equation}~\label{eq:resolvent_identity}
	(I-K(\zeta_1))^{-1}-(I-K(\zeta_2))^{-1}=
	(I-K(\zeta_1))^{-1}(K(\zeta_1)-K(\zeta_2))(I-K(\zeta_2))^{-1}
\end{equation}
for $\zeta_1,\zeta_2\in \C\sm\R_{\geq 0}$. 
Hence, for all $\zeta_1,\zeta_2$ belonging to $A\sm\R$ and compact sets $A\subset\C\sm\R_{\geq 0}$ we
have $$
  \|(I-K(\zeta_1))^{-1}-(I-K(\zeta_2))^{-1}\|_{q\to q}
  \les \|K(\zeta_1)-K(\zeta_2)\|_{q\to q}.
$$
Here we used that the operator norms of $(I-K(\zeta_1))^{-1},(I-K(\zeta_2))^{-1}$ are uniformly bounded on~$A$.
Hence, the continuity of $K$ and~\eqref{eq:resolvent_identity} imply the continuity of $\zeta\mapsto
(I-K(\zeta))^{-1}$. In the same way, the existence of a continuous extension of $\zeta\mapsto K(\zeta)$ to the positive half-axis in the operator norm topology provided by
Corollary~\ref{cor:step2} implies the existence of a continuous extension of $\zeta\mapsto (I-K(\zeta))^{-1}$
in the operator norm topology. This and Theorem~\ref{thm:Gutierrez} finally imply that $\zeta\mapsto
R(\zeta)=(I-K(\zeta))^{-1}R_0(\zeta)$ is pointwise convergent as $\zeta\to \lambda\pm i0$, $\lambda>0$ 
as claimed in Theorem~\ref{thm:main-Helmholtz-sys}. 

\medskip

Finally, using a decomposition of $V$ as in Proposition~\ref{prop:regularity} with small
$L^{\frac{n}{2}}$-part, the bound~\eqref{eq:normestimate_K} from Proposition~\ref{prop:step2-gen} implies 
$$
   \|K(\zeta)\|_{q\to q} \leq \eps + C_\eps |\zeta|^{-\frac{1}{n+1}} 
 $$ 
 for all $\eps>0,$ $\zeta\in\C\sm\R_{\geq 0}$ and some $C_\eps>0$. In particular, the norms of these
 operators tend to zero as $|\zeta|\to\infty$.  Thus, one can choose sufficiently large $R$ and sufficiently
 small $\eps$ such that $\|K(\zeta)\|_{q\to q}<\frac{1}{2}$ provided that $\zeta\in \C\setminus \R_{\geq 0}$
 and $|\zeta|\geq R.$ So the Neumann series expansion for $|\zeta|\geq R$ and the uniform boundedness for
 $|\zeta|<R$ show that $\|(I-K(\zeta))^{-1}\|\leq C<\infty$ for all $\zeta\in\C\sm\R_{\geq 0}$ and the claim
 is proved. 
\end{proof}

\begin{rem}\label{rem:real-valuedness}
  For a better understanding of the forthcoming sections, we stress here that the hypothesis of $V$ being
  Hermitian is crucial in the proof of the injectivity, thus invertibility, of $I - K(\zeta).$ Indeed, the
  vanishing of $\hat g$ on the sphere of radius $\sqrt{\lambda},$ which is a fundamental step in the proof,
  strongly relies on the identity $\Imag(\langle u,Vu\rangle)=0$, see~\eqref{eq:equation}.
\end{rem}

\section{Absence of eigenvalues for Maxwell's equations -- Proof of Theorem~\ref{thm:absenceEV}}
\label{section:absenceEV}

  In this section we prove Theorem~\ref{thm:absenceEV}. To this end we rewrite the time-harmonic Maxwell
  system~\eqref{eq:Maxwell-harmonic} as a linear Helmholtz system without first order terms so that the result
  essentially follows  from Theorem~\ref{thm:abs-emb-evs}. To write down this Helmholtz system we introduce
  some notation. For any given $\eps,\mu$ as in (A1),(A2) and $\zeta\in\C$ we define the 
  $6\times 6$ complex-valued block matrices/operators
  \begin{equation}\label{eq:matrices}
    \cV(\zeta) :=\matII{V_1(\zeta)}{-i\zeta v\times}{i\zeta v\times}{V_2(\zeta)},\quad
    \cL_1(\zeta) :=\matII{i\zeta (\eps\mu)^{\frac{1}{2}}I_3}{-\frac{1}{2}\nabla(\log\eps)\times}{
    -\frac{1}{2}\nabla(\log\mu)\times}{-i\zeta  (\eps\mu)^{\frac{1}{2}} I_3},\quad 
    \cL_2:= \matII{0}{-\nabla\times}{-\nabla\times}{0}. 
  \end{equation} 
  Here, $v:=2\nabla((\varepsilon \mu)^{\frac{1}{2}}):\R^3\to\R^3$ and $V_1(\zeta),V_2(\zeta):\R^3\to
  \C^{3\times 3}$ are given by
  \begin{equation}\label{eq:cond}
 	\begin{split}
 	&V_1(\zeta):= -\eps^{-\frac{1}{2}} \Delta(\eps^{\frac{1}{2}}) I_3
 	+\nabla \nabla^T(\log\varepsilon)   
 	- \zeta^2(\varepsilon_\infty \mu_\infty-\varepsilon \mu)I_3,\\
 	&V_2(\zeta):= - \mu^{-\frac{1}{2}} \Delta(\mu^{\frac{1}{2}}) I_3 
 	+\nabla \nabla^T(\log\mu) - \zeta^2(\varepsilon_\infty \mu_\infty-\varepsilon \mu)I_3, 
 	\end{split}
 \end{equation}
 In the following Lemma  we show that solutions of the time-harmonic Maxwell system give rise to  
 solutions of a $6\times 6$ Helmholtz system with complex-valued coefficients given
 by~\eqref{eq:matrices},\eqref{eq:cond}. 
 The proof, which is purely computational, can
 be found in Appendix A.
 
\begin{lem} \label{lem:MaxwellHelmholtz}
  Assume  (A1), $\zeta\in\C$ and let $J_e,J_m\in L^2_{\loc}(\R^3;\C^3)$ be divergence-free. 
  Then every weak solution $(E,H)\in H^1_{\loc}(\R^3;\C^6)$ of the time-harmonic Maxwell
  system~\eqref{eq:Maxwell-harmonic} satisfies
 \begin{equation}\label{eq:Helmholtz_compl}
 	(\Delta + \zeta^2\eps_\infty\mu_\infty)\vecII{\tilde E}{\tilde H} 
 	+ \cV(\zeta)\vecII{\tilde 	E}{\tilde H} 
 	= \cL_1(\zeta)\vecII{\tilde J_e}{\tilde J_m} 
 	+ \cL_2\vecII{\tilde J_e}{\tilde J_m}  \quad\text{in }\R^3
 \end{equation}
 in the weak sense where 
 $(\tilde E,\tilde H):= (\eps^{\frac{1}{2}}E,\mu^{\frac{1}{2}}H)$ and 
 $(\tilde J_e,\tilde J_m):=(\mu^{\frac{1}{2}} J_e, \eps^{\frac{1}{2}} J_m)$.
\end{lem}
\medskip

\begin{rem} ~
    Lemma~\ref{lem:MaxwellHelmholtz} allows to deduce further properties of
    solutions of time-harmonic Maxwell's equations from the corresponding theory for elliptic PDEs. For
    instance, one may deduce local regularity properties as we will do in
    Proposition~\ref{prop:regularityMaxwell}. We refer to~\cite[Section~3]{Alberti_Lectures} for 
    other approaches to regularity results for time-harmonic Maxwell's equations.
    Further features such as Harnack inequalities or maximum
    principles can be proved as well. Our assumptions on the data $\eps,\mu$ may however be far from optimal.
\end{rem}
 
  \medskip
  
  \noindent\textbf{Proof of Theorem~\ref{thm:absenceEV}:} Let $(E,H)\in H^1_{\loc}(\R^3;\C^6)$ be a
  weak solution of the  homogeneous ($J_\textup{e}= J_\textup{m}=0$) time-harmonic Maxwell
  system~\eqref{eq:Maxwell-harmonic} for $\zeta\in \C$ and assume $(1+|x|)^{\tau_1-\frac{1}{2}}(|E|+|H|)\in
  L^2(\R^3)$ for some $\tau_1>0$. From~\eqref{eq:Maxwell-harmonic} we infer 
  $(1+|x|)^{\tau_1-\frac{1}{2}}(|\nabla\times E|+|\nabla\times H|)\in L^2(\R^3)$ as well as
  $$
    \int_{\R^3} \frac{1}{\mu} (\nabla\times E) \cdot (\nabla\times \phi) \dx 
    = \zeta^2 \int_{\R^3} \eps E \phi
    \qquad\text{for all }\phi\in C_c^\infty(\R^3;\C^3).
  $$
  By density of test functions and $E\in H^1_{\loc}(\R^3;\C^3)$ we obtain ($\phi=\chi \bar E$)
  \begin{equation}\label{eq:absenceEVI}
    \int_{\R^3} \frac{1}{\mu} |\nabla\times E|^2\chi  \dx
    + \int_{\R^3} \frac{1}{\mu}  (\nabla\times E) \cdot (\nabla\chi \times \bar E) \dx
    = \zeta^2 \int_{\R^3} \eps |E|^2\chi \dx
      \qquad\text{for all }\chi\in C_c^\infty(\R^3).
  \end{equation}
  We choose $\chi=\chi^*(\cdot/R)$ where $\chi^*\in C_c^\infty(\R^3)$ is a real-valued radially nonincreasing
  nonnegative function that is identically one near the origin so that $\chi^*(\cdot/R)\nearrow 1$ as $R\to\infty$. 
  Using $|\nabla\chi^*(z)|\les (1+|z|)^{-1}$ we get for $R\geq 1$
  \begin{align*}
    \left| 
    \int_{\R^3} \frac{1}{\mu}  (\nabla\times E) \cdot (\nabla \chi  \times \bar E) \dx
    \right|
    &\les R^{-1} \int_{\R^3} |\nabla\times E| |\nabla\chi^*(x/R)| |E| \,dx\\ 
    &\les R^{-1}  \int_{\R^3} |\nabla\times E| |E| (1+|x|/R)^{-1} \,dx \\ 
	&=  \int_{\R^3} |\nabla\times E| |E| \cdot (R+|x|)^{-1} \,dx \\
    &\les R^{-2\tau_1} \int_{\R^3} |\nabla\times E| |E| (1+|x|)^{2\tau_1-1} \,dx \\
    &\les R^{-2\tau_1}   \qquad (R\to\infty). 
  \end{align*}
  In other words, the second integral in~\eqref{eq:absenceEVI} vanishes as $R\to\infty$.
  
  \medskip
  
  From this we conclude as follows. In the case $\Imag(\zeta^2)\neq 0$ we take the imaginary part
  of~\eqref{eq:absenceEVI} and get from the Monotone Convergence Theorem $\int_{\R^3} \eps|E|^2=0$,
  hence $E=0$. In the case $\Real(\zeta^2)\leq 0$ we take the real part of~\eqref{eq:absenceEVI} and obtain
  $\int_{\R^3} \frac{1}{\mu} |\nabla\times E|^2  - \Real(\zeta^2)\eps|E|^2 = 0$. Again, $E=0$.
  So we have $E=0$ in both cases, which then  implies $H=0$ because of~\eqref{eq:Maxwell-harmonic}.
  This proves the absence of eigenvalues for all $\zeta\in\C\sm\R \cup\{0\}$. 
  
  \medskip

   We now prove the claim for $\zeta\in\R\sm\{0\}$.  We deduce from Lemma~\ref{lem:MaxwellHelmholtz} that
  $(\tilde E,\tilde H)\in H^1_{\loc}(\R^3;\C^6)$ is a weak solution of the Helmholtz
  system~\eqref{eq:Helmholtz_compl} for $\tilde J_e=\tilde J_m=0$.
  After decomposing $\tilde E,\tilde H$ and the coefficient matrix $\cV(\zeta)$ into real and imaginary
  part, we find that $u:=(\Real(\tilde E),\Real(\tilde H),\Imag(\tilde E),\Imag(\tilde H))$  
  is a weak solution in $H^1_{\loc}(\R^3;\R^{12})$  of a real $(12\times 12)$-Helmholtz system
  of the form  $(\Delta + \lambda)u +Vu=0$ in~$\R^3$ where $\lambda=\zeta^2\eps_\infty\mu_\infty$ and
  $V\in L^{[\frac{3}{2},2]}(\R^3;\R^{12\times 12})$. The latter is a consequence of~(A2). 
  Since our assumptions imply $(1+|x|)^{\tau_1 -\frac{1}{2}}u\in L^2(\R^3;\R^{12})$ and 
  $\lambda=\zeta^2\eps_\infty\mu_\infty>0$, Theorem~\ref{thm:abs-emb-evs} yields $u\equiv 0$ and thus
  $E\equiv H\equiv 0.$ This finishes the proof.  \qed
  
  \begin{rem}\label{rem:absenceEV}
	 In the proof of Theorem~\ref{thm:absenceEV} we used two different approaches to treat the cases
	$\zeta \in \C \setminus \R\cup\{0\}$ and  $\zeta\in \R\setminus \{0\}.$ As a matter of fact, 
	we cannot deduce the absence of   eigenvalues $\zeta\in \C \setminus \R$
	by a reduction to the Helmholtz-type system~\eqref{eq:Helmholtz_compl} as in the case $\zeta\in\R\sm\{0\}$.	
	Indeed, as soon as we allow $\Imag(\lambda)\neq 0$ (recall that $\lambda=\zeta^2\eps_\infty \mu_\infty$), the
	function $u:=(\Real(\tilde E),\Real(\tilde H),\Imag(\tilde E),\Imag(\tilde H))$ satisfies  $(\Delta +
	\Real \lambda)u + Wu + Vu=0,$ where $W$ is a constant-valued $12\times 12$-matrix given by
	\begin{equation*} W=
		\begin{pmatrix}
			0 & \Imag (\lambda) I_6\\
			-\Imag (\lambda) I_6 & 0
		\end{pmatrix}.
	\end{equation*}
	The lack of decay of $W$ rules out the possibility of applying Theorem~\ref{thm:abs-emb-evs}  and as a
	consequence one cannot conclude $u\equiv 0.$ So the difficulty of treating complex-valued potentials cannot
	be resolved just by taking the real and imaginary parts of
	the equation. On the contrary, as we will see in the next section, the treatment of complex-valued
	potentials requires a more accurate analysis of the problem. 
\end{rem}

\section{The LAP for Maxwell's equations -- Proof of Theorem~\ref{thm:main}}
\label{sec:LAP}

This section is devoted to the proof of the Limiting Absorption Principle for the time-harmonic Maxwell
system~\eqref{eq:Maxwell-harmonic}  stated in Theorem~\ref{thm:main}. We shall first give an
overview of the main steps of the proof, the rigorous details are provided afterwards.  We start by
considering $(E_\zeta, H_\zeta)\in H^1(\R^3;\C^6)$, the uniquely determined solutions of the
\textit{approximating} time-harmonic Maxwell system
 	\begin{equation}\label{eq:approximating-Maxwell}
 		i \zeta \eps E_\zeta - \nabla \times H_\zeta=-J_\textup{e}^\zeta, \qquad
		i \zeta \mu H_\zeta + \nabla \times E_\zeta=J_\textup{m}^\zeta,
 	\end{equation}
 	where $\zeta\in\C\sm\R$ and $J_\textup{e}^\zeta,J_\textup{m}^\zeta\in L^p(\R^3;\C^3)\cap L^{\tilde p}(\R^3;\C^3)\cap
 	L^2(\R^3;\C^3)$ are divergence-free currents that approximate the given divergence-free currents $J_\textup{e},J_\textup{m}\in
 	L^p(\R^3;\C^3)\cap L^{\tilde p}(\R^3;\C^3)$ as $\zeta\to \omega\in\R_{>0}$. 
 	The necessity of introducing the
 	approximating problem~\eqref{eq:approximating-Maxwell} with square integrable currents $(J_\textup{e}^\zeta,
 	J_\textup{m}^\zeta)$ comes from the fact that, up to our knowledge, the existence and uniqueness of solutions to
 	the time-harmonic Maxwell system~\eqref{eq:Maxwell-harmonic} for general divergence-free currents
 	$(J_\textup{e}, J_\textup{m})\in L^p(\R^3;\C^6)\cap L^{\tilde p}(\R^3;\C^6)$ are not known.  	
 	In our analysis we will use the results and notation of the previous sections, notably
 	$\cV(\zeta),\cL_1(\zeta),\cL_2,V_1(\zeta),V_2(\zeta)$ and the vector field 
 	$$
 	  v = 2\nabla((\varepsilon \mu)^{\frac{1}{2}})
 	$$ 
 	from~\eqref{eq:matrices} and~\eqref{eq:cond}.
 	 
 	 \medskip
 	 
 	In order to prove the Limiting Absorption Principle, the main task is to prove the convergence of
 	the sequence $(E_\zeta, H_\zeta)$ in $L^q(\R^3;\C^6)$ as $\zeta\to \omega\pm i0$.  
 	Notice that  Sobolev's Embedding Theorem implies	$(E_\zeta, H_\zeta)\in H^1(\R^3;\C^6)\subset
 	L^q(\R^3;\C^6)$ due to  $3<q<6$. From Lemma~\ref{lem:MaxwellHelmholtz} one infers that the functions
 	\begin{equation}\label{eq:def-udelta}
 		u_\zeta:=(\tilde{E}_\zeta, \tilde{H}_\zeta):=(\eps^{\frac{1}{2}}E_\zeta,\mu^{\frac{1}{2}}H_\zeta)
 	\end{equation}
 	are solutions of the Helmholtz system  
 	\begin{equation}\label{eq:Helmholtz-approximating}
 		(\Delta + \zeta^2\eps_\infty\mu_\infty)u_\zeta
 	+ \cV(\zeta)u_\zeta
 	= \cL_1(\zeta)\tilde J^\zeta
 	+ \cL_2\tilde J^\zeta,
 	\end{equation}
 	where $\cV(\zeta), \cL_1(\zeta)$ and $\cL_2$ are defined in~\eqref{eq:matrices} and $\tilde
 	J^\zeta=(\tilde J_\textup{e}^\zeta,\tilde J_\textup{m}^\zeta):=(\mu^\frac{1}{2} J_\textup{e}^\zeta, \eps^\frac{1}{2} J_\textup{m}^\zeta).$
 	Due to the explicit relation between the spectral parameters in the Maxwell and Helmholtz
 	systems, the limiting case $\zeta=\omega
 	\pm i0,\omega\in\R\sm\{0\}$ in the Maxwell system corresponds to the limiting case in the Helmholtz system
 	$\zeta=\lambda\pm \sign(\omega)i0$, $\lambda=\omega^2 \eps_\infty\mu_\infty>0$. The boundedness assumption 
 	on $\eps, \mu$ from (A1) implies that, as soon as we are able to provide a uniform bound of $\|u_\zeta\|_q$
 	as $|\Imag(\zeta)|\to 0$, a corresponding bound also holds true for $\|(E_\zeta, H_\zeta)\|_q.$

 	\medskip
 	
 	In order to prove such bounds for $\|u_\zeta\|_q$ we need to investigate the vectorial Helmholtz type
 	operators 
 	$$
 	\Delta + \zeta^2 \eps_\infty \mu_\infty I_m + \cV(\zeta),
 	$$ 
 	that we may rewrite as $(\Delta + \zeta^2 \eps_\infty \mu_\infty)(I - \mathcal K(\zeta))$ where  
 \begin{equation}\label{eq:defK_MAxwell}
 	\begin{split}
   \mathcal K(\zeta)
   &:= - R_0\big(\zeta^2\eps_\infty\mu_\infty\big)\cV(\zeta)  \qquad\qquad\quad (\zeta\in\C\sm \R) \\
   \mathcal K(\omega\pm i0)  
   &:= - R_0\big((\omega\pm i0)^2\eps_\infty\mu_\infty\big)\cV(\omega) 
   \qquad (\omega\in\R\sm \{0\}) 
 	\end{split}
 \end{equation}
  and $(\omega\pm i0)^2 \eps_\infty\mu_\infty := \omega^2\eps_\infty\mu_\infty \pm \sign(\omega)i0$. 
  Since the potential $\cV(\zeta)$ is in general not Hermitian (see~\eqref{eq:matrices}), not even for 
  $\zeta=\omega\pm i0,\omega\in\R\sm\{0\}$, we cannot verify the sufficient condition $\Imag (\langle
  u,\cV(\zeta) u\rangle)=0$ for the  bijectivity of $I - \mathcal K(\zeta)$, see the proof of
  Corollary~\ref{cor:Helmholtz_Bijectivity} and Remark~\ref{rem:real-valuedness}.   
  Nevertheless, in Proposition~\ref{prop:InjectivityEstimate} we will quantify the potential lack of
  injectivity of the operator $I-\mathcal K(\zeta)$ through the following injectivity estimates 
\begin{equation}\label{eq:injectivity-preliminary}
	\|u\|_{q_1}+\|u\|_{q_2}   \lesssim \|(I-\mathcal K(\zeta)) u\|_{q_1}  +  \|(I-\mathcal K(\zeta))
	u\|_{q_2}  
	+  \left|\Imag(\langle u, \cV(\zeta) u \rangle) \right|^{1/2}
\end{equation}
  for all $u=(u_e,u_m)\in L^{q_1}(\R^3;\C^6)\cap L^{q_2}(\R^3;\C^6)$ and $|\Real(\zeta)|\geq \delta>0$.
 We remark that it is for proving~\eqref{eq:injectivity-preliminary} that
the additional integrability assumptions from (A3) enter  the proof of Theorem~\ref{thm:main}. 

\medskip

The estimate~\eqref{eq:injectivity-preliminary} is valid for any $u=(u_e,u_m)\in L^q(\R^3;\C^6),$ no matter
whether $u$ is a solution of the Helmholtz system~\eqref{eq:Helmholtz-approximating} or not. On the other
hand, if we consider the family $u_\zeta$ of solutions to~\eqref{eq:Helmholtz-approximating} defined
in~\eqref{eq:def-udelta}, the following representation
formula in terms of the resolvent $R_0$ of the Laplacian is available 
 \begin{equation}\label{eq:HelmholtzIntegrated}
 	(I-\mathcal{K}(\zeta))u_\zeta
 	= R_0\big(\zeta^2\eps_\infty\mu_\infty\big)\big[\cL_1(\zeta)\tilde J^\zeta+\cL_2\tilde
 	J^\zeta\big].
 \end{equation}   
	Thus the injectivity estimate~\eqref{eq:injectivity-preliminary},\eqref{eq:injectivity-preliminary2} yields
   	\begin{align}\label{eq:further-injectivity}
   	  \begin{aligned}
	 \|u_\zeta\|_{q_1}+\|u_\zeta\|_{q_2}
	 &\lesssim
	\|R_0\big(\zeta^2\eps_\infty\mu_\infty\big)\big[\cL_1(\zeta)\tilde J^\zeta+\cL_2\tilde
	J^\zeta\big]\|_{q_1}
	+\|R_0\big(\zeta^2\eps_\infty\mu_\infty\big)\big[\cL_1(\zeta)\tilde J^\zeta+\cL_2\tilde
	J^\zeta\big]\|_{q_2}\\
	& \phantom{\lesssim} + \left|\Imag(\langle u, \cV(\zeta) u \rangle) \right|^{1/2}.
	\end{aligned}
   	\end{align}
 As soon as~\eqref{eq:further-injectivity} is proved, the final step will be to bound the right-hand side
 of~\eqref{eq:further-injectivity}  in terms of $\| J\|_p + \| J\|_{\tilde p}$ and other terms depending
 on $\|u_\zeta\|_{q_1}+\|u_\zeta\|_{q_2}$ that become negligible as $\zeta\to \omega\pm i0$.
 Estimating the first two terms only requires  minor modifications of  Corollary~\ref{cor:res-free-matrix}
 and it is explicitly performed in Proposition~\ref{prop:EstimateJ}. On the other hand, in order to estimate
 the third term, the structure of  Maxwell's equations~\eqref{eq:approximating-Maxwell}   comes into play.
 Indeed, in order to get the claimed bound, we shall not only make use of the explicit
 expression of $u_\zeta$ in terms of $(E_\zeta, H_\zeta)$ from~\eqref{eq:def-udelta},
 but also that $(E_\zeta, H_\zeta)$ solves Maxwell's equations~\eqref{eq:approximating-Maxwell}.
  Combining the previous steps one gets the following uniform estimate
	\begin{equation}\label{eq:boundedness-preliminary}
		\|E_\zeta\|_q + \|H_\zeta\|_q \lesssim \|J\|_p + \|J\|_{\tilde p} + o(1) \qquad\text{as
		}\zeta\to \omega\pm i0.
	\end{equation}
From~\eqref{eq:boundedness-preliminary}, in a rather standard way (see Subsection~\ref{subsec:LAP-Maxwell}),
one obtains the Limiting Absorption Principle contained in Theorem~\ref{thm:main}.
	   
 \medskip
 
 In the following $p,\tilde p,q,q_1,q_2$ are chosen as in Theorem~\ref{thm:main}. 
 We will use the notation from~\eqref{eq:matrices} and~\eqref{eq:cond}. In the proofs we will write 
 $L^s:=L^s(\R^3;\C^6)$ and similarly for $W^{1,s},L^s_{\loc}$ etc.

 \subsection{Injectivity estimates}
 
  First we recall from Proposition~\ref{prop:regularity} (with
  $\kappa=\frac{n}{2}=\frac{3}{2},\tilde\kappa=\frac{n+1}{2}=2$) the regularity and integrability properties
  of $L^q(\R^3;\C^6)$-solutions to $(I-\mathcal K(\zeta))u=0$.
  Using the definitions~\eqref{eq:matrices},\eqref{eq:defK_MAxwell} and assumption (A2) we get the following
  result.
  
\begin{prop}\label{prop:regularityMaxwell}
  Assume (A2) and let $q$ satisfy~$3<q<6$. Moreover assume $(I-\mathcal
  K(\zeta))u=0$ for some $u \in L^q(\R^3;\C^6)$ where $\zeta\in\C\sm\R$ or $\zeta=\omega\pm
  i0,$ $\omega\in\R\sm\{0\}$. Then any solution $u\in L^q(\R^3;\C^6)$ of $(I-\mathcal K(\zeta))u=0$ belongs to 
  $L^r(\R^3;\C^6)\cap H^1_{\loc}(\R^3;\C^6)$ for all $r\in (3,\infty)$. 
  Moreover, for any given such $r,q$ we have $\|u\|_r\les  \|u\|_q$. 
\end{prop}
 
 As in the Helmholtz case these integrability properties are
 actually better for $\zeta\in\C\sm\R$ where we even have $u\in H^1(\R^3;\C^6)$. 
 Next we present the crucial scalar condition ensuring the injectivity of 
 $I-\mathcal K(\zeta)=I+R_0(\zeta^2\eps_\infty\mu_\infty)[\cV(\zeta) \cdot]$. 
 It comes as no surprise that the condition to guarantee injectivity only involves $\Imag(\langle u,
 \cV(\zeta) u \rangle)$. It can be seen as an isotropic inhomogeneous variant of the
 Silver-M\"uller type radiation condition, which, roughly speaking, represents the analogue for Maxwell
 equations of the well-known Sommerfeld radiation conditions in the context of Helmholtz
 equations.
   
\begin{prop}\label{prop:injectivityCondition}
  Let (A2) hold and assume $3<q<6$.
  Moreover assume $(I-\mathcal K(\zeta))u=0$ for some $u=(u_e,u_m) \in L^q(\R^3;\C^6)$
  where $\zeta\in\C$ with $\Imag(\zeta^2)\neq 0$ or $\zeta=\omega\pm i0,$ $\omega\in\R\sm\{0\}$. Then
  \begin{equation} \label{eq:InjectivityCondition}
     \Imag(\langle u, \cV(\zeta) u \rangle)
    \qquad\Longleftrightarrow\qquad
    u= 0.
  \end{equation}
\end{prop}
\begin{proof}
  We only need to prove the implication ``$\Longrightarrow$'', so we assume that 
  $\Imag(\langle u, \cV(\zeta) u \rangle)=0$. This implies
  \begin{equation*}
  	\begin{split}
    0 
    &= \Imag\left(\int_{\R^3} \ov{u}\cdot \cV(\zeta) u \dx\right)  \\
    &= \Imag\left(\int_{\R^3} \ov{\mathcal K(\zeta)u}\cdot \cV(\zeta) u \dx\right) \\
    &\stackrel{\eqref{eq:defK_MAxwell}}= -\Imag\left(\int_{\R^3}
    \ov{R_0(\zeta^2\eps_\infty\mu_\infty)[\cV(\zeta)u]}\cdot \cV(\zeta) u \dx \right).
    \end{split}
  \end{equation*}
  For $\Imag(\zeta^2)\neq 0$ we deduce as in equation~(3.6) of~\cite{GolSch_LAP} that 
  $\cV(\zeta)u\equiv 0$ and hence $u\equiv \mathcal K(\zeta)u\equiv 0$ by~\eqref{eq:defK_MAxwell}.
  In the case $\zeta=\omega\pm i0,\omega\in\R\sm\{0\}$ 
  we get as in the proof of Corollary~\ref{cor:Helmholtz_Bijectivity}
  for $\lambda=\omega^2\eps_\infty\mu_\infty>0$
  \begin{align*}
    0
    &= \Imag\left(\int_{\R^3} \ov{\mathcal K(\omega\pm i0)u}\cdot \cV(\omega) u\right) \\
    &= -\Imag\left(\int_{\R^3} \ov{R_0(\lambda\pm \sign(\omega)i0)[\cV(\omega)u]}\cdot \cV(\omega)
    u\right) \\
    &= \mp \sign(\omega) c \int_{S_\lambda} |\widehat{\cV(\omega) u}|^2 \,d\sigma_\lambda
  \end{align*}    
  for some $c\neq 0$. Hence, $\widehat{\cV(\omega) u}=0$ on $S_\lambda$ in the
  $L^2$-trace sense. Moreover, $\cV(\omega)u\in L^p$ for some (sufficiently large) $p\in
  [1,\frac{4}{3})$ by assumption (A2) and Proposition~\ref{prop:regularityMaxwell}. So
  Proposition~\ref{prop:decay} implies for some $\tau_1>0$
  $$
   (1+|\cdot|)^{\tau_1-\frac{1}{2}} u
   = (1+|\cdot|)^{\tau_1-\frac{1}{2}} \mathcal K(\omega\pm i0)u
   = -(1+|\cdot|)^{\tau_1-\frac{1}{2}} R_0(\lambda\pm\sign(\omega)i0)[\cV(\omega)u] \in L^2.
  $$ 
  Given that $u$ solves the homogeneous Helmholtz system~\eqref{eq:ev-pb}  we deduce from
  Theorem~\ref{thm:abs-emb-evs} $u= 0$.
\end{proof}
  
  From this fact we deduce our injectivity estimates. We take the condition
  $\Imag(\zeta^2)\neq 0$ or $\zeta=\omega\pm i0, \omega\in\R\sm\{0\}$ from the previous proposition into
  account by restricting our attention to spectral parameters $\zeta$ with nontrivial real parts. Recall
  from Theorem~\ref{thm:Gutierrez} and Proposition~\ref{prop:step2-gen} that this implies 
  continuity properties of $\zeta\mapsto \mathcal K(\zeta)$.

\begin{prop}\label{prop:InjectivityEstimate}
  Assume (A2),(A3). Then, for any given compact subset $K\subset \{\zeta\in
  \C:\Real(\zeta)\neq 0\}$ we have for all $\zeta\in K\sm\R$ and all $u=(u_e,u_m)\in L^{q_1}(\R^3;\C^6)\cap
  L^{q_2}(\R^3;\C^6)$  
  \begin{align*}
    \|u\|_{q_1} + \|u\|_{q_2}
    &\les  \|(I-\mathcal K(\zeta))u\|_{q_1}  + \|(I-\mathcal K(\zeta))u\|_{q_2} 
    +|\Imag(\langle u, \cV(\zeta) u \rangle)|^{1/2}.
  \end{align*}
\end{prop}
\begin{proof}
  We argue by contradiction and assume that there are sequences $(\zeta^j)\subset K$ and $(u^j)\subset
  L^{q_1}\cap L^{q_2}$ such that $\|u^j\|_{q_1}+\|u^j\|_{q_2}=1$ and
  \begin{align}\label{eq:InjectivityEstimateAssumption}
    \begin{aligned}
    \|(I-\mathcal K(\zeta^j))u^j\|_{q_1} + 
    \|(I-\mathcal K(\zeta^j))u^j\|_{q_2} \to 0, \quad
    \Imag(\langle u^j, \cV(\zeta^j) u^j \rangle)
    \to 0. 
  \end{aligned}
  \end{align}
  We choose subsequences such that  $\zeta^j\to \zeta^*\in K$ and $u^j\wto u^*$ in $L^{q_1}$ and
  $L^{q_2}$.
  In the case $\zeta^*\in\R\sm\{0\},\Imag(\zeta^j)\to 0^+$ we will write $\mathcal K(\zeta^*)$ instead of
  $\mathcal K(\zeta^*+ i0)$ for notational simplicity. Clearly similar arguments apply in the case
  $\Imag(\zeta^j)\to 0^-$.
   The second part of Corollary~\ref{cor:step2} and~\eqref{eq:defK_MAxwell} imply $\mathcal
  K(\zeta^j)u^j\to \mathcal K(\zeta^*)u^*$ so that $\|(I-\mathcal K(\zeta^j))u^j\|_{q_1}+\|(I-\mathcal
  K(\zeta^j))u^j\|_{q_2}\to 0$ gives $(I- \mathcal K(\zeta^*))u^*=0$ and thus $u^j\to u^*$ in $L^q_1$ and in
  $L^{q_2}$.
   From the second
   part of~\eqref{eq:InjectivityEstimateAssumption} we want to deduce that the injectivity
   condition~\eqref{eq:InjectivityCondition} holds for $(u^*,\zeta^*)$. So we need to show
   $$
     \Imag(\langle u^*, \cV(\zeta^*) u^* \rangle
     = \lim_{j\to\infty} \Imag(\langle u^j, \cV(\zeta^j) u^j \rangle
     = 0.
   $$
   To verify the first equality we use formula
   \begin{align} \label{eq:injectivity-preliminary2}
     \begin{aligned}
    &\Imag\left(\int_{\R^3} \ov{u}\cdot \cV(\zeta) u \dx\right) \\
    &= \Imag \left(\int_{\R^3} \ov u_e V_1(\zeta) u_e + \ov u_m V_2(\zeta) u_m  
      + i\zeta \big(\ov u_e\cdot (v\times u_m)- \ov u_m\cdot (v\times u_e)\big) \dx \right)\\
    &= - \Imag(\zeta^2) \int_{\R^3} (\eps_\infty\mu_\infty-\eps\mu)|u|^2 \dx - 
    \Imag \left(i\zeta \int_{\R^3} v\cdot (u_m\times \ov u_e) - v\cdot (u_e\times \ov u_m)  \dx \right)   \\
    &= \Imag(\zeta^2) \int_{\R^3} (\eps\mu-\eps_\infty\mu_\infty)|u|^2\dx
    - \Imag \left(i\zeta \int_{\R^3} v\cdot  2\Real(u_m\times \ov u_e) \dx \right)    \\
    &=  \Imag(\zeta^2) \int_{\R^3} (\eps\mu-\eps_\infty\mu_\infty)|u|^2 \dx - 2\Real(\zeta) \int_{\R^3} v\cdot
    \Real(u_m\times \ov u_e) \dx.  
  \end{aligned}
  \end{align}
   We write 
   $|\eps_\infty\mu_\infty-\eps\mu|+|v|=m_1+m_2$ where 
   $m_1\in L^{\frac{q_1}{q_1-2}}(\R^3),m_2\in L^{\frac{q_2}{q_2-2}}(\R^3)$. This is possible due to 
   assumption (A3).  
  \begin{align*}
     &\left|\Imag( (\zeta^j)^2)\int_{\R^3} (\eps\mu-\eps_\infty\mu_\infty)|u^j|^2 \dx - \Imag(
    (\zeta^*)^2)\int_{\R^3} (\eps\mu-\eps_\infty\mu_\infty)|u^*|^2\dx \right| \\  
     &\leq  |(\zeta^j)^2-(\zeta^*)^2|\int_{\R^3} (m_1+m_2) |u^j|^2  \dx
     + |\zeta^*|^2\int_{\R^3} (m_1+m_2) ||u^j|^2-|u^*|^2| \dx \\
     &\les |(\zeta^j)^2-(\zeta^*)^2| 
     (\|m_1\|_{\frac{q_1}{q_1-2}} \|u^j\|_{q_1}^2 + \|m_2\|_{\frac{q_2}{q-2}} \|u^j\|_{q_2}^2) \\
     &+ \|m_1\|_{\frac{q_1}{q_1-2}} \|u^j-u\|_{q_1}\||u^j|+|u|\|_{q_1}  
      + \|m_2\|_{\frac{q_2}{q_2-2}} \|u^j-u\|_{q_2}\||u^j|+|u|\|_{q_2}\\
      &\les   (\|m_1\|_{\frac{q_1}{q_1-2}} + \|m_2\|_{\frac{q_2}{q_2-2}})
       (|(\zeta^j)^2-(\zeta^*)^2|+\|u_j-u\|_{q_1} +\|u_j-u\|_{q_2})  \\
      &= o(1)\qquad (j\to\infty).
  \end{align*}
  Analogous computations yield 
  $$
    \Real(\zeta^j) \int_{\R^3} v\cdot \Real(u_m^j\times \ov u_e^j) \dx
    \to \Real(\zeta^*) \int_{\R^3} v\cdot  \Real(u_m^*\times \ov u_e^*) \dx
    \qquad (j\to\infty). 
  $$
  So~\eqref{eq:InjectivityCondition} holds and Proposition~\ref{prop:injectivityCondition} implies $u^*=0$.
  This however contradicts $u^j\to u^*=0$ and $\|u^j\|_{q_1}+\|u^j\|_{q_2}=1$. So the assumption was false,
  which proves the claim.
\end{proof}

 \subsection{Bounds for $E_\zeta,H_\zeta$}

  Proposition~\ref{prop:InjectivityEstimate} makes it possible to  bound the $L^q$-norm of solutions
  $u_\zeta:=(u_e^\zeta,u_m^\zeta):=(\tilde{E}_\zeta, \tilde{H}_\zeta)$ of the Helmholtz
  system~\eqref{eq:Helmholtz_compl} with $\zeta\in\C\sm\R$
  in terms of $J$  as soon as we find suitable bounds for 
 $\Imag(\langle u_\zeta, \cV(\zeta) u_\zeta \rangle).$
 Those are provided in the next proposition.
 
\begin{prop}\label{prop:ExtraCondition}
  Let the assumptions (A1),(A2),(A3)  hold. Then, for any given $\zeta\in\C\sm\R$ the solutions 
  $u_\zeta:=(u_e^\zeta, u_m^\zeta):=(\tilde E_\zeta,\tilde H_\zeta):=(\eps^\frac{1}{2} E_\zeta,
  \mu^\frac{1}{2} H_\zeta)\in L^q(\R^3;\C^6)$ of the Helmholtz system~\eqref{eq:Helmholtz_compl} satisfy
  \begin{align*}
    \begin{aligned}
    \int_{\R^3} v\cdot \Real(u_m^\zeta\times \ov{u_e^\zeta})   \dx
    &=  \Imag(\zeta) \int_{\R^3}  (\eps\mu-\eps_\infty\mu_\infty) |u_\zeta|^2  \dx \\
    &+ \int_{\R^3}  (\eps\mu-\eps_\infty\mu_\infty)\,  \Real\left(\mu^{-\frac{1}{2}}\ov{J_\textup{m}^\zeta}\cdot
    u_m^\zeta - \eps^{-\frac{1}{2}}J_\textup{e}^\zeta\cdot \ov{u_e^\zeta}\dx\right).
     \end{aligned}  
 \end{align*}
  In particular, for $\zeta\in K\sm \R$ and any compact set $K\subset \C$,
 \begin{align}\label{eq:extraConditionII}
   \begin{aligned}
    |\Imag(\langle u_\zeta, \cV(\zeta) u_\zeta \rangle)|
    \les |\Imag(\zeta)|
    (\|u_\zeta\|_{q_1}+\|u_\zeta\|_{q_2})^2 + (\|J_\textup{e}^\zeta\|_p+\|J_\textup{m}^\zeta\|_p)(\|u_\zeta\|_{q_1}+\|u_\zeta\|_{q_2}).
    \end{aligned}
 \end{align}
\end{prop}
\begin{proof}
   We recall from~\eqref{eq:cond} the identity
   $$
     v
     = 2\nabla((\eps\mu)^{1/2})
     = (\eps\mu)^{-\frac{1}{2}} \nabla(\eps\mu) 
     = (\eps\mu)^{-\frac{1}{2}} \nabla(\eps\mu-\eps_\infty\mu_\infty) .
   $$
   Then integration by parts gives 
  \begin{align*}
    &\int_{\R^3} v\cdot \Real(u_m^\zeta\times \ov{u_e^\zeta}) \dx \\
    &=  \int_{\R^3} \nabla (\eps\mu-\eps_\infty\mu_\infty)\cdot \Real\left( H_\zeta\times \ov{ E_\zeta}
    \dx\right) \\
    &=  -\int_{\R^3}  (\eps\mu-\eps_\infty\mu_\infty)\, \left[\nabla\cdot \Real\left( H_\zeta\times
    \ov{ E_\zeta}\right)\right]  \dx  \\
    &=  -\int_{\R^3} (\eps\mu-\eps_\infty\mu_\infty)\, \Real\left( (\nabla \times  H_\zeta)\cdot \ov { E_\zeta} 
    - (\ov{\nabla \times   E_\zeta})\cdot  H_\zeta\right)  \dx \\
    &= - \int_{\R^3} (\eps\mu-\eps_\infty\mu_\infty)\, \Real\left( (i\zeta \eps E_\zeta+J_\textup{e}^\zeta)\cdot \ov E_\zeta 
    - \ov{(-i\zeta \mu H_\zeta+J_\textup{m}^\zeta)}\cdot H_\zeta\right) \dx \\
    &= - \int_{\R^3}  (\eps\mu-\eps_\infty\mu_\infty)\, \Real\left( i\zeta \eps|E_\zeta|^2 + J_\textup{e}^\zeta\cdot \ov {E_\zeta} 
    - i \ov{\zeta} \mu |H_\zeta|^2  - \ov{J_\textup{m}^\zeta}\cdot H_\zeta \right)  \dx \\
    &= \Imag(\zeta) \int_{\R^3}  (\eps\mu-\eps_\infty\mu_\infty)\,  \left(\eps|E_\zeta|^2  
     + \mu |H_\zeta|^2\right) \dx
     + \int_{\R^3}  (\eps\mu-\eps_\infty\mu_\infty)\, \Real \left( \ov{J_\textup{m}^\zeta}\cdot H_\zeta - 
     J_\textup{e}^\zeta\cdot \ov {E_\zeta}  \dx \right)   \\
    &=  \Imag(\zeta) \int_{\R^3}  (\eps\mu-\eps_\infty\mu_\infty) |u_\zeta|^2 \dx
     + \int_{\R^3}  (\eps\mu-\eps_\infty\mu_\infty)\,  \Real\left(\mu^{-\frac{1}{2}}\ov{J_\textup{m}^\zeta}\cdot u_m^\zeta - 
     \eps^{-\frac{1}{2}}J_\textup{e}^\zeta\cdot \ov{u_e^\zeta} \right) \dx. 
  \end{align*}
  So  H\"older's inequality and (A3) imply with $m_1,m_2$ as in the proof of
  Proposition~\ref{prop:InjectivityEstimate}  
  \begin{align*}
    \left|2\Real(\zeta)  
    \int_{\R^3} v\cdot \Real(u_m^\zeta\times \ov{u_e^\zeta}) \dx \right| 
    &\les  |\Imag(\zeta)| (\|m_1\|_{\frac{q_1}{q_1-2}}\|u_\zeta\|_{q_1}^2 +
    \|m_2\|_{\frac{q_2}{q_2-2}}\|u_\zeta\|_{q_2}^2) \\
    &+  \int_{\R^3}  (m_1+m_2) (|J_\textup{m}^\zeta| |u_m^\zeta| + |J_\textup{e}^\zeta| |u_e^\zeta|) \dx
 \end{align*} 
 We have $m_1\in L^{\frac{q_1}{q_1-2}}(\R^3)\cap L^\infty(\R^3)$
 and $m_2\in L^{\frac{q_2}{q_2-2}}(\R^3)\cap L^\infty(\R^3)$ by (A1),
 $J_\textup{m}^\zeta,J_\textup{e}^\zeta\in L^p\cap L^{\tilde p}$ and $u_m^\zeta,u_e^\zeta\in L^{q_1}\cap L^{q_2}$. 
 Our assumptions on $p,\tilde p,q_1,q_2$ from the theorem imply 
 \begin{align*}
   \frac{1}{\infty}+\frac{1}{\tilde p}+\frac{1}{q_1}&\leq 1\leq \frac{q_1-2}{q_1}+\frac{1}{p}+\frac{1}{q_1},
   \\
   \frac{1}{\infty}+\frac{1}{\tilde p}+\frac{1}{q_2}&\leq 1\leq \frac{q_2-2}{q_2}+\frac{1}{p}+\frac{1}{q_2}.
 \end{align*}
 So we can find suitable exponents for  H\"older's inequality and thus obtain 
 \begin{align*}
    \left|2\Real(\zeta)  
    \int_{\R^3} v\cdot \Real(u_m^\zeta\times \ov{u_e^\zeta}) \dx \right| 
    &\les  |\Imag(\zeta)| (\|u_\zeta\|_{q_1} + \|u_\zeta\|_{q_2})^2 \\
    &+  (\|J_\textup{e}^\zeta \|_p+\|J_\textup{m}^\zeta\|_p+
    \|J_\textup{e}^\zeta \|_{\tilde p}+\|J_\textup{m}^\zeta\|_{\tilde p})(\|u_\zeta\|_{q_1} + \|u_\zeta\|_{q_2}).
  \end{align*} 
 Moreover, $\Imag(\zeta^2)=2\Real(\zeta)\Imag(\zeta)$ gives  
 $$
   \left| 
    \Imag( \zeta^2)\int_{\R^3} (\eps\mu-\eps_\infty\mu_\infty)|u_\zeta|^2 \dx\right| 
    \les  |\Imag(\zeta)| (\|u_\zeta\|_{q_1}+\|u_\zeta\|_{q_2})^2  
 $$
  so that~\eqref{eq:extraConditionII} is proved in view of~\eqref{eq:injectivity-preliminary2}. 
\end{proof}

  Combining this fact and the injectivity estimates from Proposition~\ref{prop:InjectivityEstimate} we obtain
   uniform  bounds for the solutions $(E_\zeta,H_\zeta)$ provided that $|\Imag(\zeta)|$ is sufficiently
   small.

\begin{cor} \label{cor:Maxwellbounds}
  Let the assumptions (A1),(A2),(A3) hold and let  $K\subset \C$ be
  compact. Then, for $|\Imag(\zeta)|$ sufficiently small, any solution $(E_\zeta,H_\zeta)$ of the
  time-harmonic Maxwell system~\eqref{eq:Maxwell-harmonic} for $\zeta\in K\sm \R$ satisfies 
  \begin{align}\label{eq:ControlEHdelta}
    \begin{aligned}
    \|E_\zeta\|_q   + \|H_\zeta\|_q  
    &\les\|R_0(\zeta^2\eps_\infty\mu_\infty)[\mathcal L_1(\zeta)\tilde J^\zeta
    + \mathcal L_2\tilde J^\zeta ]\|_{q_1} \\
    &+ \|R_0(\zeta^2\eps_\infty\mu_\infty)[\mathcal L_1(\zeta)\tilde J^\zeta
    + \mathcal L_2\tilde J^\zeta ]\|_{q_2}   + \|J^\zeta\|_p + \|J^\zeta\|_{\tilde p}. 
    \end{aligned} 
  \end{align}
\end{cor} 
\begin{proof}
  We define $u_\zeta:= (\tilde E_{\zeta},\tilde
  H_\zeta):=(\eps^{\frac{1}{2}}E_\zeta,\mu^{\frac{1}{2}}H_\zeta)$.
  By Lemma~\ref{lem:MaxwellHelmholtz} these functions solve the Helmholtz system~\eqref{eq:Helmholtz_compl}
  and hence satisfy the representation formula~\eqref{eq:HelmholtzIntegrated}.
  So Proposition~\ref{prop:InjectivityEstimate} and Proposition~\ref{prop:ExtraCondition} give
  \begin{align*}
     \|u_\zeta\|_{q_1} + \|u_\zeta\|_{q_2} 
    &\les  \|(I-\mathcal K(\zeta))u_\zeta\|_{q_1} + \|(I-\mathcal K(\zeta))u_\zeta\|_{q_2}
    + |\Imag(\langle u_\zeta, \cV(\zeta) u_\zeta \rangle)|^{1/2} \\
    &\les  \|R_0( \zeta^2\eps_\infty\mu_\infty) \big[\mathcal L_1(\zeta)\tilde J^\zeta 
      + \mathcal L_2\tilde J^\zeta\big]\|_{q_1} 
      + \|R_0( \zeta^2\eps_\infty\mu_\infty) \big[\mathcal L_1(\zeta)\tilde J^\zeta 
      + \mathcal L_2\tilde J^\zeta\big]\|_{q_2} \\
    &+   \sqrt{|\Imag(\zeta)|}  (\|u_\zeta\|_{q_1}+\|u_\zeta\|_{q_2})  + (\|J_\textup{e}^\zeta \|_p +
    \|J_\textup{m}^\zeta\|_p)^{\frac{1}{2}} (\|u_\zeta\|_{q_1}+\|u_\zeta\|_{q_2})^{\frac{1}{2}}. 
  \end{align*}
  This and $\|u_\zeta\|_q \les \|u_\zeta\|_{q_1} + \|u_\zeta\|_{q_2}$
  yields the corresponding bound for $u_\zeta$ provided that
  $|\Imag(\zeta)|$ is sufficiently small. Assumption~(A1) implies
  $\|E_\zeta\|_r  +  \|H_\zeta\|_r\les \|u_\zeta\|_r$ for $r\in\{q_1,q_2\}$ and~\eqref{eq:ControlEHdelta}
  follows.
\end{proof}

\subsection{Proof of the Limiting Absorption Principle}\label{subsec:LAP-Maxwell}
  
  We first prove the existence of the functions $(E_\zeta,H_\zeta)$ the bounds for which we provided above.
  We recall that it is defined as the unique solution in $H^1(\R^3;\C^6)\subset L^q(\R^3;\C^6)$ of
  the time-harmonic Maxwell system~\eqref{eq:Maxwell-harmonic}  with divergence-free
  currents $J_\textup{e}^\zeta,J_\textup{m}^\zeta$ lying in $L^p(\R^3;\C^3)\cap L^{\tilde p}(\R^3;\C^3)\cap L^2(\R^3;\C^3)$
  that converge to $J_\textup{e},J_\textup{m}$, respectively. (The reason for considering $J_\textup{e}^\zeta,J_\textup{m}^\zeta$ instead of
  $J_\textup{e},J_\textup{m}$ is because the existence of $L^q(\R^3;\C^6)$-solutions $(E_\zeta,H_\zeta)$ for the currents
  $J_\textup{e}, J_\textup{m}\in L^p(\R^3;\C^3)\cap L^{\tilde p}(\R^3;\C^3)$ is not clear.)  
  In the next proposition we first show that divergence-free currents $J_\textup{e}, J_\textup{m}\in
  L^p(\R^3; \C^3)\cap L^{\tilde p}(\R^3;\C^3)$ can be approximated by a sequence of divergence-free
  currents $J_\textup{e}^\zeta, J_\textup{m}^\zeta\in L^p(\R^3; \C^3)\cap L^{\tilde p}(\R^3;\C^3)\cap L^2(\R^3;\C^3)$.
  
  \begin{prop}\label{prop:Approximation}
      Let $p,\tilde p\in (1,\infty)$ and assume $J_\textup{e}, J_\textup{m}\in L^p(\R^3;\C^3)\cap L^{\tilde p}(\R^3;\C^3)$ to be divergence-free. 
      Then there are divergence-free vector fields $J_\textup{e}^\zeta, J_\textup{m}^\zeta\in
      L^p(\R^3;\C^3)\cap L^{\tilde p}(\R^3;\C^3)\cap L^2(\R^3;\C^3)$ satisfying $(J_\textup{e}^\zeta,J_\textup{m}^\zeta)\to
      (J_\textup{e},J_\textup{m})$ in $L^p(\R^3;\C^3)\cap L^{\tilde p}(\R^3;\C^3)$ as $\zeta\to
      \omega\in\R\sm\{0\}$.
  \end{prop}
  \begin{proof}
    The map $\Pi\colon f\mapsto f- \cF^{-1}(|\xi|^{-2}\xi\xi^T\hat f)$ is a bounded linear operator from
    $L^p\cap L^{\tilde p}$ to the divergence-free vector fields in $L^p\cap
    L^{\tilde p}$. This is a consequence of the $L^p\cap L^{\tilde p}$-boundedness of
    Riesz transforms. So for any given $f\in L^p\cap L^{\tilde p}$ 
    we can choose $(f_n)\subset \mathcal S$ such that $f_n$ converges to $f$ in $L^p
    \cap L^{\tilde p}$. The sequence $(\Pi f_n)$ then has the desired properties.
  \end{proof}
  
  Next we show that for $J_\textup{e}^\zeta,J_\textup{m}^\zeta$ as in Proposition~\ref{prop:Approximation} there are
  uniquely determined solutions $(E_\zeta,H_\zeta)$ in $H^1(\R^3;\C^6)$.  In the proof  we will need the following result for $r=2$.
 
 \begin{prop} \label{prop:GradientEstimates}
    Assume (A1) and $\zeta\in\C,r\in (1,\infty)$. Then every solution of~\eqref{eq:Maxwell-harmonic} satisfies  
    $$
     \|\nabla E\|_r + \|\nabla H\|_r
      \les (1+|\zeta|)(\|E\|_r + \|H\|_r)+ \|J_\textup{e}\|_r+\|J_\textup{m}\|_r 
    $$
  \end{prop}
  \begin{proof}
    Since $D=\eps E$ and $B=\mu H$ are divergence-free, we have 
    $$
     \nabla \cdot E = \eps^{-1}\nabla \eps\cdot E,\qquad 
     \nabla \cdot H = \mu^{-1}\nabla \mu\cdot H.
    $$
    This and~\cite[Theorem~1.1]{vonWahl_Estimating} imply
  \begin{align*}
    \|\nabla E\|_r + \|\nabla H\|_r
    &\les \|\nabla\times E\|_r + \|\nabla\cdot E\|_r + \|\nabla\times H\|_r + \|\nabla\cdot H\|_r\\
    &\les \|-i\zeta \mu H + J_\textup{m}\|_r + \|\eps^{-1}\nabla \eps\cdot E\|_r 
     + \|i\zeta \eps E+ J_\textup{e}\|_r + \|\mu^{-1}\nabla \mu\cdot H\|_r \\
    &\les  (1+|\zeta|)(\|E\|_r + \|H\|_r) + \|J_\textup{e}\|_r+\|J_\textup{m}\|_r.
  \end{align*} 
  \end{proof}

  \begin{prop}\label{prop:EdeltaHdelta}
     Assume (A1). Then, for $\zeta\in\C$ with $\Real(\zeta),\Imag(\zeta)\neq 0$, there is a unique solution
     $(E_\zeta,H_\zeta)\in H^1(\R^3;\C^6)$ of~\eqref{eq:Maxwell-harmonic} for the divergence-free currents
     given by $J_\textup{e}^\zeta,J_\textup{m}^\zeta\in L^2(\R^3;\C^3)$.
  \end{prop}
  \begin{proof}
    The existence and uniqueness of such a solution $(E_\zeta,H_\zeta)\in 
    H(\curl;\R^3)\times
    H(\curl;\R^3)$ can be proved as in~\cite[Section 7.4]{FabMor_Electromagnetism}.  
    From $\eps,\mu\in W^{1,\infty}(\R^3)$ by (A1) and Proposition~\ref{prop:GradientEstimates} for $r=2$ we
    obtain $(E_\zeta,H_\zeta)\in H^1(\R^3;\C^6)$.
  \end{proof}
  
  The preceding propositions ensure that the sequences of  solutions $(E_\zeta,H_\zeta)$
  we were speaking of really exist in the space $H^1(\R^3;\C^6)$ and  in particular in $L^q(\R^3;\C^6)$ for
  all $q\in (3,6)$  by Sobolev's Embedding Theorem. In
  Corollary~\ref{cor:Maxwellbounds} we showed that $(E_\zeta,H_\zeta)$ remain bounded once we have bounds for suitable Lebesgue-norms of  
   $\tilde J^\zeta$ and $R_0(\zeta^2\eps_\infty\mu_\infty)\big[\mathcal L_1(\zeta)\tilde J^\zeta
    + \mathcal L_2\tilde J^\zeta \big]$ 
 which are independent of $\Imag(\zeta)$. As mentioned earlier, this can be achieved rather easily
 with the aid of Theorem~\ref{thm:Gutierrez} and a suitable modification of it when first order
 derivates are involved, see Theorem~\ref{thm:resolvent-der} in the Appendix.
 
 \begin{prop} \label{prop:EstimateJ}
   Assume (A1) and let $K\subset \{\zeta\in\C:\Real(\zeta)\neq 0\}$ be compact. Then, for
  $\zeta\in K\sm\R$ and $\tilde J^\zeta:=(\mu^{\frac{1}{2}}J_\textup{e}^\zeta,\eps^{\frac{1}{2}}J_\textup{m}^\zeta)$  
   as above, we have   
   $$
    \|R_0( \zeta^2\eps_\infty\mu_\infty)\big[\mathcal L_1(\zeta)\tilde
    J+\mathcal L_2\tilde J^\zeta \big]\|_{q_1}
    +\|R_0( \zeta^2\eps_\infty\mu_\infty)\big[\mathcal L_1(\zeta)\tilde
    J+\mathcal L_2\tilde J^\zeta \big]\|_{q_2}  
    \les \|J^\zeta\|_p + \|J^\zeta\|_{\tilde p}.
  $$
 \end{prop}
 \begin{proof}
   To bound the term involving $\cL_1$ we use Theorem~\ref{thm:Gutierrez}. Since $\|\mathcal
   L_1(\zeta)\|_\infty\les 1+|\zeta|\les 1$ by the definition of $\cL_1$ from~\eqref{eq:cond}
   and assumption (A1) we get
   \begin{align*}
     \|R_0( \zeta^2\eps_\infty\mu_\infty) \big[\mathcal L_1(\zeta)\tilde J^\zeta  \big]\|_{q_1}  
     \les  \|\mathcal L_1(\zeta)\tilde J^\zeta\|_p 
     \les   \|\tilde J^\zeta\|_p 
     \les \|J^\zeta\|_p.
   \end{align*}
    The estimate for the term involving $\cL_2$ corresponds to the special case $n=3,m=6$ in
    Theorem~\ref{thm:resolvent-der}:
    \begin{align*}
     \|R_0( \zeta^2\eps_\infty\mu_\infty) \big[\mathcal L_2\tilde J^\zeta \big]\|_{q_1} 
     &\les   \|\tilde J^\zeta\|_p  + \|\tilde J^\zeta\|_{\tilde p}
     \les \|J^\zeta\|_p + \|J^\zeta\|_{\tilde p}.
   \end{align*}
   Since the same holds for $q_1$ replaced by $q_2$, this proves the claim.
 \end{proof}

  \medskip
  
  Now we are in the position to prove the Limiting Absorption Principle for time-harmonic Maxwell's
  equations~\eqref{eq:Maxwell-harmonic}.
  
  \medskip 
  
  \noindent\textbf{Proof of Theorem~\ref{thm:main}:}\; In order to prove Theorem~\ref{thm:main}, it suffices
  to combine the auxiliary results that we established above. So assume (A1),(A2),(A3) and let $p,q$ 
  and $J_\textup{e}, J_\textup{m}\in L^p\cap  L^{\tilde p}$ be  given  as in the theorem.
  We prove the existence of the solutions $(E_\omega^\pm,H_\omega^\pm)$ with the desired properties by
  proving the convergence of the solutions $(E_\zeta,H_\zeta)$ as outlined in part~(i) of the
  theorem. To reduce the notation we only consider the limit $\zeta\to \omega+i0$.
  
  \medskip
  
  \textit{Proof of (i):}\; For $\zeta\in\C\sm\R$ with $\Imag(\zeta)>0,\Real(\zeta)\neq 0$ let 
  $J_\textup{e}^\zeta,J_\textup{m}^\zeta\in L^p\cap
  L^{\tilde p}\cap L^2$ be the divergence-free approximating sequence whose existence
  is ensured by Proposition~\ref{prop:Approximation}.
  Let then $(E_\zeta,H_\zeta)$ denote the unique $H^1$-solutions of the corresponding
  inhomogeneous time-harmonic Maxwell system~\eqref{eq:Maxwell-harmonic} from Proposition~\ref{prop:EdeltaHdelta}. 
  Corollary~\ref{cor:Maxwellbounds} yields for small $|\Imag(\zeta)|$
  \begin{align*}
    \|E_\zeta\|_q +  \|H_\zeta\|_{q}
   &\les\|R_0(\zeta^2\eps_\infty\mu_\infty)\big[\mathcal L_1(\zeta)\tilde J^\zeta
    + \mathcal L_2\tilde J^\zeta \big]\|_{q_1}
    + \|R_0(\zeta^2\eps_\infty\mu_\infty)\big[\mathcal L_1(\zeta)\tilde J^\zeta
    + \mathcal L_2\tilde J^\zeta \big]\|_{q_2}
    + \|J^\zeta\|_p
    + \|J^\zeta\|_{\tilde p}    
  \end{align*}
  where $\tilde J^\zeta :=(\mu^{\frac{1}{2}}J_\textup{e}^\zeta,\eps^{\frac{1}{2}}J_\textup{m}^\zeta)$. So
  Proposition~\ref{prop:EstimateJ} implies
  $$
    \|E_\zeta\|_q  + \|H_\zeta\|_{q} 
    \les \|J^\zeta\|_p + \|J^\zeta\|_{\tilde p}
    = \|J\|_p + \|J\|_{\tilde p} + o(1) \quad \text{as }\zeta\to \omega+i0,
  $$ 
  which proves that the sequence of approximate solutions $(E_\zeta,H_\zeta)$ is bounded in
  $L^q$. So a subsequence of $(u_\zeta)$ defined via $u_\zeta:= (\tilde E_\zeta,\tilde
  H_\zeta):=(\eps^\frac{1}{2} E_\zeta,\mu^\frac{1}{2}H_\zeta)$ converges weakly to some 
  $u_\omega^+ =:(\tilde E_\omega^+,\tilde H_\omega^+)$ in $L^q$. Defining 
  $(E_\omega^+,H_\omega^+):= (\eps^{-\frac{1}{2}} \tilde E_\omega^+,\mu^{-\frac{1}{2}}\tilde H_\omega^+)$  we thus
  obtain a weak solution of the time-harmonic Maxwell system~\eqref{eq:Maxwell-harmonic} (for $\zeta=\omega$)
  that satisfies 
  $$ 
    \|E_\omega^+\|_q  + \|H_\omega^+\|_{q} 
    \les  \|\tilde E_\omega^+\|_q  + \|\tilde H_\omega^+\|_{q}
    \les \|J\|_p + \|J\|_{\tilde p}.
  $$
  In the first estimate assumption (A1) is used. This proves the existence of the solution
  $(E_\omega^+,H_\omega^+)$ along with the corresponding norm estimate. To conclude the proof of (i) we need
  to show that for any given approximations $J_\textup{e}^\zeta,J_\textup{m}^\zeta$ as above the full sequence
  $(u_\zeta)$ converges  to $u_\omega^+$.
    
  \medskip
    
  So let $(\zeta_j),(\tilde \zeta_j)$ sequences converging to $\omega+i0$  and
  let $J^{1,\zeta_j},J^{2,\tilde\zeta_j}$ be divergence-free currents converging to $J$. Let
  $u_\omega^1,u_\omega^2\in L^q$ denote the corresponding weak limits, i.e., $u_{\zeta_j}\wto
  u_\omega^1,u_{\tilde\zeta_j}\wto u_\omega^2$. We need to show $u_\omega^1=u_\omega^2$.
  From~\eqref{eq:HelmholtzIntegrated} we infer $(I-\mathcal K(\zeta_j))u_{\zeta_j} =
  f^1_{\zeta_j}$ and $(I-\mathcal K(\tilde\zeta_j))u_{\tilde\zeta_j} = f^2_{\tilde\zeta_j}$ where
  \begin{align*}
    f_\zeta^1 
 	&:= R_0\big(\zeta^2\eps_\infty\mu_\infty\big)\big[\cL_1(\zeta)\tilde
 	J^{1,\zeta}+\cL_2\tilde J^{1,\zeta}\big], \\
 	f_{\zeta}^2 
 	&:= R_0\big(\zeta^2\eps_\infty\mu_\infty\big)\big[\cL_1(\zeta)\tilde
 	J^{2,\zeta}+\cL_2\tilde J^{2,\zeta}\big], \\
  	f_{\omega}^+  
  	&:= R_0\big((\omega+i0)^2\eps_\infty\mu_\infty\big)\big[\cL_1(\omega)\tilde
  	J+\cL_2\tilde J\big]. 
  \end{align*}
  Using first Proposition~\ref{prop:EstimateJ} and $\tilde J^{1,\zeta_j}\to \tilde J$, then
  the uniform boundedness of the resolvents $R_0(\zeta_j)$ and $\mathcal L_1(\zeta_j)\to\mathcal L_1(\omega)$
  in $L^\infty$, and finally the pointwise convergence $R_0(\zeta_j^2\eps_\infty\mu_\infty)\to
  R_0((\omega+i0)^2\eps_\infty\mu_\infty)$ we infer
  \begin{align*}
    f_{\zeta_j}^1 
 	&= R_0\big(\zeta_j^2\eps_\infty\mu_\infty\big)\big[\cL_1(\zeta_j)\tilde J^{1,\zeta_j}+\cL_2\tilde
 	J^{1,\zeta_j} \big] \\
 	&= R_0\big(\zeta_j^2\eps_\infty\mu_\infty\big)\big[\cL_1(\zeta_j)\tilde J +\cL_2\tilde	J \big] + o(1) \\
 	&= R_0\big(\zeta_j^2\eps_\infty\mu_\infty\big)\big[\cL_1(\omega)\tilde J+\cL_2\tilde J\big] + o(1) \\
 	&= R_0\big((\omega+i0)^2\eps_\infty\mu_\infty\big)\big[\cL_1(\omega)\tilde J+\cL_2\tilde J\big] + o(1) \\ 
 	&= f_\omega^+ + o(1). 
  \end{align*}
  Since the same argument applies to $f_{\tilde\zeta_j}^2$, we get $f_{\zeta_j}^1-f^2_{\tilde\zeta_j}\to
  f_\omega^+-f_\omega^+ = 0$. Using the continuity and compactness properties of $\mathcal K$ from
  Corollary~\ref{cor:step2} we infer that $w:= u_\omega^1-u_\omega^2$ satisfies 
  \begin{align*}
    w &= u_{\zeta_j}-u_{\tilde\zeta_j} + o_w(1) \\
     &= \mathcal K(\zeta_j)u_{\zeta_j} -\mathcal K(\zeta_j)u_{\tilde\zeta_j} 
       + f^1_{\zeta_j} - f^2_{\tilde\zeta_j}  + o_w(1) \\
     &= \mathcal K(\omega+i0)u_\omega^1 -\mathcal K(\omega+i0)u_\omega^2       + o_w(1) \\
     &= \mathcal K(\omega+i0)w    + o_w(1).  
  \end{align*} 
  Here, $o_w(1)$ stands for a null sequence in the weak topology in $L^q$.
  The function  $(\tilde w_\textup{e},\tilde w_\textup{m})$ given by 
  $(\eps^{1/2}\tilde w_\textup{e},\mu^{1/2}\tilde w_\textup{m}):=  w$ is a weak solution of the
  homogeneous time-harmonic Maxwell system~\eqref{eq:Maxwell-harmonic} for $\zeta=\omega$  and $\tilde
  J=0$. Repeating the computations in Proposition~\ref{prop:ExtraCondition} in the limiting case $\Imag(\zeta)=0$ one finds
  \begin{align*}
     2\omega\int_{\R^3} v\cdot \Real(w_m\times \ov{w_e}) \dx = 0.
  \end{align*}
  So Proposition~\ref{prop:injectivityCondition} gives $w=0$. This proves that all possible weak
  limits coincide. Hence, the  standard subsequence-of-subsequence argument ensures that all approximating
  sequences weakly converge to the same limit as $\zeta\to \omega+i0$.
    
  \medskip
  
  To finish the proof of (i) it remains to show that this convergence also holds in the strong sense
  and hence, by elliptic regularity theory, in $H^1_{\loc}(\R^3;\C^6)$. To
  this end we recall from above $(I-\mathcal K(\zeta_j))u_{\zeta_j}=f_{\zeta_j}$. From the definition of
  $\mathcal K$ and the second part of Proposition~\ref{prop:step2-gen} we get $\mathcal K(\zeta_j)u_{\zeta_j}\to \mathcal
  K(\omega+i0)u_{\omega}^+$ in $L^q$. Moreover, we showed above  $f_{\zeta_j}\to f_\omega^+$ in
  $L^q$ as $j\to\infty$. Hence, we conclude that $(u_{\zeta_j})$ converges in $L^q$. Since the limit necessarily
  coincides with the weak limit, we finally obtain $u_{\zeta_j}\to u_\omega^+$ as $\zeta_j\to \omega+i0$. This
  proves (i) as well as 
  \begin{equation}\label{eq:EquationForSolution}
     (I-\mathcal K(\omega\pm i0))u_\omega^\pm 
     = R_0\big((\omega\pm i0)^2\eps_\infty\mu_\infty\big)\big[\cL_1(\omega)\tilde J+\cL_2\tilde J\big].
  \end{equation}
  In particular, $u_\omega^+$ solves the Helmholtz system mentioned in the theorem.
  
  \medskip
  
  We finally prove part~(iii) of the theorem, so we assume that the divergence-free currents
  $(J_\textup{e},J_\textup{m})$ lie in the smaller space $L^p \cap L^q \subset L^p \cap L^{\tilde p}$ (because
  $p<\tilde p<q$).  From above we get a solution $(E,H)\in L^q$ of~\eqref{eq:Maxwell-harmonic} which,
  according to Proposition~\ref{prop:GradientEstimates}, satisfies  
  \begin{align*}
     \|\nabla E\|_q + \|\nabla H\|_q
     \les (1+|\zeta|) (\|E\|_q + \|H\|_q) + \|J_\textup{e}\|_q+\|J_\textup{m}\|_q
     \les \|J\|_p  +  \|J\|_q. 
  \end{align*}
  Hence we conclude $E,H\in W^{1,q}(\R^3;\C^3)$ as claimed.
  \qed

\section*{Appendix A: Proof of Lemma~\ref{lem:MaxwellHelmholtz}}
 This appendix is devoted to the proof of Lemma~\ref{lem:MaxwellHelmholtz} and, in particular, to showing the validity of~\eqref{eq:Helmholtz_compl}.
For notational simplicity we verify~\eqref{eq:Helmholtz_compl} in the pointwise sense assuming that
  classical derivatives exist. This carries over to weak solutions by moving first order derivatives to the
  test functions. So let $(E,H)$ denote a weak solution of~\eqref{eq:Maxwell-harmonic} as assumed.
  Then  $(\tilde E,\tilde H) =(\eps^{-\frac{1}{2}}D,\mu^{-\frac{1}{2}}B)$ where  $D,B$ are
  divergence-free vector fields. The latter follows from the fact that $J_e,J_m$ are divergence-free.  
  So we have  $\nabla \times \nabla \times D=\nabla (\nabla \cdot D)-\Delta D=-\Delta D$ and obtain
   \begin{align} \label{eq:LaplacianTildeE}
     \begin{aligned}
     \Delta\tilde E
     &= \Delta (\eps^{-\frac{1}{2}}) D 
       + 2 [\nabla(\eps^{-\frac{1}{2}})\cdot\nabla ] D
       + \eps^{-\frac{1}{2}} \Delta D \\
     &= \left(\frac{3}{4}\eps^{-2}|\nabla\eps|^2 - \frac{1}{2}\eps^{-1}\Delta\eps \right)
       \tilde E  -  \eps^{-\frac{3}{2}} [\nabla \eps\cdot\nabla ] D
       - \eps^{-\frac{1}{2}} \nabla\times\nabla\times D.
      \end{aligned}      
   \end{align}
   The second order term can be simplified with the aid of~\eqref{eq:Maxwell-harmonic}. The vector calculus
   identities
   \begin{align*}
	  \nabla \times (\psi A) &= \nabla \psi \times A+\psi (\nabla \times A),\\
	  \nabla \times (A\times C)&=A(\nabla \cdot C) - C(\nabla \cdot A) + (C\cdot \nabla)A - (A\cdot \nabla)C
   \end{align*} 
   for scalar fields $\psi$ and vector fields $A,C$ lead to  
   \begin{align*}
     \nabla\times\nabla\times D
     &= \nabla\times \nabla\times (\eps E) \\
     &= \nabla\times \left( \nabla\eps \times E + \eps (\nabla\times E)\right) \\
     &= \nabla\times \left( \nabla\eps \times E + \eps \left[-i\zeta\mu
     H + J_m \right] \right)
     \\
     &= \nabla\times \left( \nabla\eps \times E\right)   
       - i\zeta \nabla(\eps\mu) \times H 
       - i\zeta \eps\mu (\nabla\times H)
       + \nabla\times (\eps J_m)   \\
     &= \nabla \eps (\nabla\cdot E) - (\Delta \eps) E 
     + (E\cdot\nabla ) \nabla \eps - (\nabla\eps\cdot \nabla) E \\
     &-i\zeta \nabla(\eps\mu) \times H 
       - i\zeta \eps\mu (i\zeta \eps E+J_e)
       + \nabla\times (\eps^{\frac{1}{2}}\tilde J_m).     
   \end{align*}
   To simplify these terms we use that $D=\eps E$ is divergence-free and thus $\nabla\cdot E =
   -\eps^{-1}\nabla\eps\cdot E$. Moreover,
   $$
      (\nabla\eps\cdot\nabla)E
     =   (\nabla\eps\cdot\nabla) (\eps^{-1}D)
     = -\eps^{-2}|\nabla\eps|^2 D + \eps^{-1}(\nabla\eps\cdot\nabla)D.
   $$ 
   This implies
   \begin{align*}
     -\eps^{-\frac{1}{2}}\nabla\times\nabla\times D
     &= -\eps^{-\frac{1}{2}} \cdot \left[ - \eps^{-1} 
     \nabla \eps (\nabla \eps\cdot E) -  (\Delta \eps) E 
     + (E\cdot\nabla ) \nabla \eps 
     + \eps^{-2}|\nabla\eps|^2 D - \eps^{-1}(\nabla\eps\cdot\nabla)D
     \right]   \\
     &+i\zeta (\eps\mu)^{-\frac{1}{2}} \nabla(\eps\mu) \times \tilde H 
       -  \zeta^2 \eps\mu \tilde E  
       + i\zeta (\eps\mu)^{\frac{1}{2}} \tilde J_e 
       - \eps^{-\frac{1}{2}} \nabla\times  (\eps^{\frac{1}{2}}\tilde J_m) \\
     &=   \eps^{-2} \nabla \eps (\nabla \eps)^T \tilde E 
   	+  \eps^{-1}(\Delta \eps) \tilde E 
     - \eps^{-1}(\tilde E\cdot\nabla ) \nabla \eps   
     - \eps^{-2} |\nabla\eps|^2 \tilde E 
     + \eps^{-\frac{3}{2}} (\nabla\eps\cdot \nabla) D     \\
     &+ i\zeta (\eps\mu)^{-\frac{1}{2}} \nabla(\eps\mu) \times \tilde H 
       -  \zeta^2 \eps\mu \tilde E  
       + i\zeta (\eps\mu)^{\frac{1}{2}} \tilde J_e 
       -   \nabla\times \tilde J_m 
       - \frac{1}{2} \eps^{-1} \nabla\eps \times \tilde J_m. 
   \end{align*}
   Combining this formula with~\eqref{eq:LaplacianTildeE} we find that the first order terms (involving $D$)
   cancel and
   \begin{align*}
     \Delta\tilde E
     &= \left(\frac{3}{4}\eps^{-2}|\nabla\eps|^2 - \frac{1}{2}\eps^{-1}\Delta\eps \right)
       \tilde E  + \eps^{-2} \nabla \eps (\nabla \eps)^T \tilde E 
     +  \eps^{-1} \Delta \eps \tilde E 
      - \eps^{-1} (\tilde E\cdot\nabla ) \nabla \eps   
     -  \eps^{-2} |\nabla\eps|^2 \tilde E \\
      &+ i\zeta (\eps\mu)^{-\frac{1}{2}} \nabla(\eps\mu) \times \tilde H 
       -  \zeta^2 \eps\mu \tilde E  
       + i\zeta (\eps\mu)^{\frac{1}{2}} \tilde J_e 
       -   \nabla\times  \tilde J_m 
       - \frac{1}{2} \eps^{-1} \nabla\eps \times \tilde J_m \\
     &= \left(-\frac{1}{4}\eps^{-2}|\nabla\eps|^2 + \frac{1}{2}\eps^{-1}\Delta\eps 
     + \eps^{-2} \nabla \eps (\nabla \eps)^T
     -  \zeta^2 \eps\mu
     - \eps^{-1} \nabla \nabla^T \eps
     \right)  \tilde E \\  
       &+ i\zeta (\eps\mu)^{-\frac{1}{2}} \nabla(\eps\mu) \times \tilde H 
       + i\zeta (\eps\mu)^{\frac{1}{2}} \tilde J_e 
       -   \nabla\times  \tilde J_m 
       - \frac{1}{2} \eps^{-1} \nabla\eps \times \tilde J_m \\
     &= \Big[
     \eps^{-\frac{1}{2}} \Delta(\eps^{\frac{1}{2}})
 	-\nabla \nabla^T(\log\varepsilon)   
 	+ \zeta^2(\varepsilon_\infty \mu_\infty-\varepsilon \mu)\Big] \tilde E 
 	- \zeta^2\eps_\infty\mu_\infty \tilde E\\
 	 &+ 2i\zeta  \nabla((\varepsilon \mu)^{\frac{1}{2}}) \times \tilde H 
       + i\zeta (\eps\mu)^{\frac{1}{2}} \tilde J_e 
       -   \nabla\times  \tilde J_m 
       - \frac{1}{2} \nabla (\log\eps)  \times \tilde J_m.  
   \end{align*}
   This corresponds to the first line in~\eqref{eq:Helmholtz_compl}. To derive the second line, one
   proceeds in an analogous manner and subsequently derives the formulas 
   \begin{align*}
     \Delta\tilde H
     &= \left(\frac{3}{4}\mu^{-2}|\nabla\mu|^2 - \frac{1}{2}\mu^{-1}\Delta\mu \right)
       \tilde H  -  \mu^{-\frac{3}{2}} [\nabla \mu\cdot\nabla ] B
       - \mu^{-\frac{1}{2}} \nabla\times\nabla\times B, \\
     \nabla\times\nabla\times B
     &= \nabla \mu (\nabla\cdot H) - (\Delta \mu) H 
     + (H\cdot\nabla ) \nabla \mu - (\nabla\mu\cdot \nabla) H \\
     & + i\zeta \nabla(\eps\mu) \times E 
       + i\zeta \eps\mu (-i\zeta \mu H+J_m)
       + \nabla\times (\mu^{\frac{1}{2}}\tilde J_e), \\
      -\mu^{-\frac{1}{2}}\nabla\times\nabla\times B
     &= \mu^{-2} \nabla \mu (\nabla \mu)^T \tilde H 
   	  +  \mu^{-1}(\Delta \mu) \tilde H 
     - \mu^{-1}(\tilde H\cdot\nabla ) \nabla \mu   
     - \mu^{-2} |\nabla\mu|^2 \tilde H 
     + \mu^{-\frac{3}{2}} (\nabla\mu\cdot \nabla) B     \\
     & - i\zeta (\eps\mu)^{-\frac{1}{2}} \nabla(\eps\mu) \times \tilde E 
       -  \zeta^2 \eps\mu \tilde H  
       - i\zeta (\eps\mu)^{\frac{1}{2}} \tilde J_m 
       -   \nabla\times \tilde J_e 
       - \frac{1}{2} \mu^{-1} \nabla\mu \times \tilde J_e, \\
     \Delta\tilde H
     &=  \Big[ \mu^{-\frac{1}{2}} \Delta(\mu^{\frac{1}{2}})
 	-\nabla \nabla^T(\log\mu)   
 	+ \zeta^2(\varepsilon_\infty \mu_\infty-\varepsilon \mu)\Big] \tilde H 
 	-\zeta^2\eps_\infty\mu_\infty\tilde H \\
 	 &- 2i\zeta  \nabla((\varepsilon \mu)^{\frac{1}{2}}) \times \tilde E 
       - i\zeta (\eps\mu)^{\frac{1}{2}} \tilde J_m 
       -   \nabla\times  \tilde J_e 
       - \frac{1}{2} \nabla (\log\mu) \times \tilde J_e.       
   \end{align*}\qed

\section*{Appendix B: Uniform estimates for $R_0(\zeta) \partial_j$}

This appendix is devoted to the proof of Theorem~\ref{thm:resolvent-der} (see below) that we needed in the
proof of the Limiting Absorption Principle for Maxwell's equations.
In order to do that we need the following classical result on Fourier multipliers.
   
   \begin{thm}[Mikhlin-H\"ormander] \label{thm:MikhlinHoermander}
     Let $n\in\N,1<r<\infty$. For $k:=\lfloor \frac{n}{2}\rfloor +1$ assume that $m\in C^k(\R^n)$
     satisfies $|\partial_\alpha m(\xi)| \leq A|\xi|^{-|\alpha|}$ for all multi-indices $\alpha\in\N_0^n$ with
     $|\alpha|\leq k$.
     Then  
     $$
       \| \cF^{-1}(m\cF f)\|_r \leq C_{n,r} A \|f\|_r
     $$
   \end{thm}  
   \begin{proof}
   		The result is a particular case of Theorem 6.2.7 in~\cite{Graf_Classical}.
   \end{proof}
   
   We also need some boundedness properties of Bessel potentials. 
   
   \begin{thm}\label{thm:Bessel}
   		Assume $n\in\N,n\geq 2$.
   		Let $\mathcal J_1$ be the Bessel potential of order $1$ defined as
   		\begin{equation*}
   			\mathcal{J}_1 f:=\cF^{-1} \left( \frac{1}{\sqrt{1+ |\xi|^2}} \hat f \right),
   			\qquad \text{alternatively}, \qquad
   			\mathcal{J}_1 f:=G\ast f:=\cF^{-1}\left(\frac{1}{\sqrt{1+|\xi|^2}} \right)\ast f,
   		\end{equation*}
   		Assume $1\leq p\leq q\leq \infty$ be such that 
   		\begin{equation*}
   			0\leq \frac{1}{p}-\frac{1}{q}\leq \frac{1}{n},
   			\qquad \left(\frac{1}{p}, \frac{1}{q} \right) \notin \left \{ \left(1, 1-\frac{1}{n} \right), \left( \frac{1}{n}, 0\right) \right\}.
   		\end{equation*}
   		Then $\mathcal J_1$ is a bounded operator from $L^p(\R^n)$ to $L^q(\R^n).$ 
   \end{thm}
   \begin{proof}
   		From~\cite[Corollary 1.2.6 (a),(b)]{Graf_Modern} and interpolation we have
   		\begin{equation*}
   			\mathcal J_1\colon L^p(\R^n)\to L^q(\R^n), 
   			\qquad\text{if}\; 1<p\leq q<\infty,\quad  0\leq \frac{1}{p}-\frac{1}{q}\leq \frac{1}{n}.
   		\end{equation*}
   		Thus, it remains to study the case $p=1$ and the case $q=\infty.$ Let us start with the case $p=1.$
   		Again from~\cite[Corollary 1.2.6 (b)]{Graf_Modern} we have that $\mathcal{J}_1\colon L^1(\R^n)\to
   		L^{q,\infty}(\R^n),$ when $p=1$ and for $\frac{1}{p}-\frac{1}{q}=\frac{1}{n}.$ Moreover,
   		from~\cite[Corollary 1.2.6 (a)]{Graf_Modern}, we know that $\mathcal{J}_1\colon L^1(\R^n)\to L^1(\R^n).$
   		Thus, Marcinkiewicz's interpolation theorem gives that $\mathcal{J}_1\colon L^1(\R^n)\to L^q(\R^n),$ with
   		$1-\frac{1}{q}<\frac{1}{n}.$   		
   		We continue with the case $q=\infty.$ We know from the proof of~\cite[Corollary 1.2.6~(b)]{Graf_Modern}
   		that the kernel $G$ satisfies $|G(x)|\les |x|^{1-n}$ if $|x|\leq 2$ and $|G(x)|\les e^{-\frac{|x|}{2}}$ if $|x|\geq 2.$
   		Using Young's convolution inequality we thus get
   		\begin{equation*}
   			\|\mathcal{J}_1 f \|_\infty
   			=\|G\ast f\|_\infty
   			\leq \|G\|_{p'} \|f \|_p, \qquad \frac{1}{p} + \frac{1}{p'}=1.
   		\end{equation*}
   		The proof is concluded once one observes that $\|G \|_{p'}<\infty$ for $0\leq \frac{1}{p}<\frac{1}{n}.$  
   \end{proof}
   
  With these results at hands, we are in the position to prove the following estimates that complement
  Theorem~\ref{thm:Gutierrez}.
  
  \begin{thm}\label{thm:resolvent-der}
  	Let $m,n\in \N, n\geq 3$ and assume $\zeta\in \C\setminus \R_{\geq 0}.$ Then, for $1\leq p,\tilde p,q\leq
  	\infty$ such that 
  	\begin{equation}\label{eq:ass-deriv}
  	\begin{split}
  	 1&\geq  \frac{1}{p}>\frac{n+1}{2n},\qquad  0\leq   \frac{1}{q}<\frac{n-1}{2n}, \qquad 
      \frac{2}{n+1} \leq \frac{1}{p}-\frac{1}{q}\leq   \frac{2}{n},\\    
	0&\leq \frac{1}{\tilde p}-\frac{1}{q}\leq \frac{1}{n},\qquad
      \left(\frac{1}{\tilde p},\frac{1}{q}\right)\notin 
      \left\{ \left(1,1-\frac{1}{n} \right),  \left(\frac{1}{n},0\right)\right\},
      \end{split}
  	\end{equation}
  	$R_0(\zeta) \partial_j f$ is a bounded linear operator from $L^p(\R^n;\C^m)$ to $L^q(\R^n;\C^m)$ satisfying 
  	\begin{equation}\label{eq:resolvent-der}
  		\|R_0(\zeta) \partial_j f\|_q\lesssim |\zeta|^{\frac{n}{2}(\frac{1}{ p}-\frac{1}{q}-\frac{1}{n})}
  		\|f\|_{p} + |\zeta|^{\frac{n}{2}(\frac{1}{\tilde p}-\frac{1}{q}-\frac{1}{n})} \|f\|_{\tilde p},
  		\qquad j=1,2,\dots, n.
  	\end{equation}
  	Moreover, there are bounded linear operators $R_0(\lambda\pm i0)\partial_j:L^p(\R^n;\C^m)\to
  	L^q(\R^n;\C^m)$ such that $R_0(\zeta)\partial_j f\to R_0(\lambda\pm i0)\partial_j f,$ $j=1,2,\dots,n$ as $\zeta\to\lambda \pm
  	i0,\lambda\in\R_{>0}$ for all $f\in L^p(\R^n;\C^m)$ and
\begin{equation*}
	\|R_0(\lambda \pm i 0)\partial_jf\|_q
	\les |\lambda|^{\frac{n}{2}(\frac{1}{ p}-\frac{1}{q}-\frac{1}{n})} \|f\|_{p} + |\lambda|^{\frac{n}{2}(\frac{1}{\tilde p}-\frac{1}{q}-\frac{1}{n})} \|f\|_{\tilde p}
	\qquad (\lambda>0).
\end{equation*}
  \end{thm}
  \begin{proof}
  	We first isolate the singularity (in Fourier space) of the Fourier multiplier
  	$\frac{1}{|\xi|^2-\zeta}$. In order to do that we introduce the cut-off function $\chi\in C^\infty_0(\R^n)$
  	with $\chi(\xi)=1$ whenever $|\xi|\leq 2$ and we define $\chi_\zeta(\xi):=\chi(|\zeta|^{-\frac{1}{2}}\xi).$ We then write
  	\begin{equation}\label{eq:is-sing}
  		R_0(\zeta) (\partial_j f)= \cF^{-1}\left( \frac{\chi_\zeta (\xi) (i\xi_j \hat{f}(\xi))}{|\xi|^2 - \zeta} \right) 
  		+ \cF^{-1}\left(\frac{(1-\chi_\zeta (\xi)) (i\xi_j \hat{f}(\xi))}{|\xi|^2 - \zeta} \right)
  		\qquad j=1,2,\dots,n.
  	\end{equation}
  	Observe that $\chi_\zeta$ is nontrivial in a neighborhood of the sphere of radius $\zeta,$ on the contrary $1-\chi_\zeta$ vanishes in the 
  	same neighborhood.  In other words, the singularity of the multiplier affects only the first term of the
  	right-hand side of~\eqref{eq:is-sing}.
  	The latter can be estimated with the aid of Theorem~\ref{thm:Gutierrez}. More specifically, one has
  	\begin{equation*}
  		\begin{split}
  			\left \| \cF^{-1} \left(\frac{\chi_\zeta (\xi) (i\xi_j \hat{f}(\xi))}{|\xi|^2 - \zeta} \right) \right\|_q
  			&= |\zeta|^{-\frac{1}{2} -\frac{n}{2q}}
       \left\|\cF^{-1}\left(  \frac{\chi(\xi)\xi_j \cF
       [f(|\zeta|^{-1/2}\cdot)](\xi)}{|\xi|^2-\frac{\zeta}{|\zeta|}} \right)\right\|_q \\
      &\les |\zeta|^{-\frac{1}{2} -\frac{n}{2q}}
       \left\|\cF^{-1}\left(
      \chi(\xi)\xi_j  \cF  [f(|\zeta|^{-1/2}\cdot)](\xi) \right) \right\|_p   \\
      &\les |\zeta|^{-\frac{1}{2} -\frac{n}{2q}}
       \left\|\cF^{-1}\left(\chi(\xi)\xi_j\right) \ast f(|\zeta|^{-1/2}\cdot)  \right\|_p    \\
      &\les |\zeta|^{-\frac{1}{2} -\frac{n}{2q}} \|f(|\zeta|^{-1/2}\cdot)\|_p    \\
      &\les |\zeta|^{\frac{n}{2}(\frac{1}{p}-\frac{1}{q}-\frac{1}{n})} \|f\|_p.
  		\end{split}
  	\end{equation*}
  	 Here we used Young's convolution inequality and that $\cF^{-1}(\chi(\xi)\xi_j)$ is
    integrable (being a Schwartz function).
    
    \medskip
    
    We now turn  to the estimate of the second term in the sum in~\eqref{eq:is-sing}. We shall use that
    \begin{equation*}
    	m(\xi):=\frac{(1-\chi(\xi)) \xi_j \sqrt{1+ |\xi|^2} }{|\xi|^2-\zeta/|\zeta|}, \qquad j=1,2,\dots, n,
    \end{equation*}  
    is a $L^r(\R^n)-L^r(\R^n)$ multiplier for $1<r<\infty,$ due to the simplified version of the
    Mikhlin-H\"ormander Theorem stated in Theorem~\ref{thm:MikhlinHoermander}. Recall that $\chi$ satisfies
    $1-\chi\equiv 0$ on a neighborhood of the unit sphere. Notice that, by~\eqref{eq:ass-deriv} we have
    $0<\frac{1}{q}<1$ or $0<\frac{1}{\tilde p}<1.$ In the case $0<\frac{1}{q}<1$ we use the above observation
    for $r=q$ and Theorem~\ref{thm:Bessel}  implies
    \begin{equation*}
    	\begin{split}
    		\cF^{-1}\left(\frac{(1-\chi_\zeta (\xi)) (i\xi_j \hat{f}(\xi))}{|\xi|^2 - \zeta} \right)
    		   &=|\zeta|^{-\frac{1}{2}-\frac{n}{2q}} \left\|\cF^{-1} \left( \frac{(1-\chi(\xi)) \xi_j
       \cF [f(|\zeta|^{-1/2}\cdot)](\xi)}{|\xi|^2-1}   \right)\right\|_q \\
       &=  |\zeta|^{-\frac{1}{2}-\frac{n}{2q}} \left\| \cF^{-1} \left( m(\xi)
      \cF\left(
      \cF^{-1}\left( 
      \frac{1}{\sqrt{|\xi|^2+1}}  \cF [f(|\zeta|^{-1/2}\cdot)](\xi)\right)
      \right)\right)\right\|_q  \\
      &=  |\zeta|^{-\frac{1}{2}-\frac{n}{2q}} \left\| \cF^{-1} \left( m(\xi)
      \cF\left(\mathcal{J}_1 [f(|\zeta|^{-1/2}\cdot)]\right) \right)\right\|_q  \\
      &\les |\zeta|^{-\frac{1}{2}-\frac{n}{2q}} \left\|\mathcal{J}_1 [f(|\zeta|^{-1/2}\cdot)]
        \right\|_q   \\
      &\les |\zeta|^{-\frac{1}{2}-\frac{n}{2q}}  \|f(|\zeta|^{-1/2}\cdot) \|_{\tilde p}   \\
      &\leq  |\zeta|^{\frac{n}{2}(\frac{1}{\tilde p}-\frac{1}{q}-\frac{1}{n})} \|f\|_{\tilde p}.
    	\end{split}
    \end{equation*}
    In the complementary case $0<\frac{1}{\tilde p}<1,$ we use the Mikhlin-H\"ormander
    Theorem~\ref{thm:MikhlinHoermander} for $r=\tilde p$ and proceed similarly.     
    Plugging the two previous bounds in~\eqref{eq:is-sing} gives~\eqref{eq:resolvent-der}. The final part of
    Theorem~\ref{thm:resolvent-der} can be proved as in Theorem~\ref{thm:Gutierrez}.
  \end{proof}
   
\section*{Acknowledgements}

We thank R. Schnaubelt (KIT, Karlsruhe) and P. D'Ancona (La Sapienza, Rome) for  
sharing their results and ideas with us. Funded by the Deutsche Forschungsgemeinschaft (DFG, German
Research Foundation) -- Project-ID 258734477 -- SFB 1173.

\bibliographystyle{plain}

\end{document}